\documentclass[a4paper,11pt]{amsart}

\DeclareFontFamily{U}{mathx}{}
\DeclareFontShape{U}{mathx}{m}{n}{<-> mathx10}{}
\DeclareSymbolFont{mathx}{U}{mathx}{m}{n}
\DeclareMathAccent{\widehat}{0}{mathx}{"70}
\DeclareMathAccent{\widecheck}{0}{mathx}{"71}
\DeclareMathAccent{\widetilde}{0}{mathx}{"72}

\usepackage{amsfonts,amsmath,amssymb}
\usepackage{mathtools}
\usepackage{xargs,xifthen}
\usepackage{xcolor}
\usepackage{graphics}
\graphicspath{ {Images/} }
\usepackage{standalone}
\usepackage{tikz}
\usepackage{pgfplots}
\usepgfplotslibrary[groupplots]
\usetikzlibrary{arrows.meta}
\usepackage[font=small,labelfont=bf]{subcaption}
\usepackage{enumitem}
\usepackage[hidelinks]{hyperref}
\usepackage[noabbrev]{cleveref}

\usepackage{macros}

\newtheorem{theorem}{Theorem}

\newtheorem{corollary}[theorem]{Corollary}

\newtheorem{lemma}[theorem]{Lemma}
\theoremstyle{definition}
\newtheorem{remark}[theorem]{Remark}

\newtheorem{example}{Example}
\theoremstyle{remark}

\newtheorem{assumption}[theorem]{Assumption}
\numberwithin{theorem}{section}
\numberwithin{equation}{section}
\numberwithin{table}{section}
\numberwithin{figure}{section}
\allowdisplaybreaks

\usepackage{bm}

\usepackage{amsopn}

\usepackage{scalerel}
\newcommand{\invdag}{\scalerel*{\rotatebox[origin=c]{180}{$\dagger$}}{\dagger}}

\begin{document}

\title[Higher-order convergence for parabolic multiscale problems]{
Optimal higher-order convergence rates for parabolic multiscale problems}
\author[B.~Kalyanaraman, F.~Krumbiegel, R.~Maier, S.~Wang]{Balaje~Kalyanaraman$^\ddagger$$^\mathsection$, Felix~Krumbiegel$^{\dagger\invdag}$, Roland~Maier$^\dagger$, and Siyang~Wang$^\ddagger$}
\address{${}^{\ddagger}$ Department of Mathematics and Mathematical Statistics, Ume\aa{} University,  901 87 Ume\aa, Sweden.}
\email{\{balaje.kalyanaraman,siyang.wang\}@umu.se}
\address{${}^{\mathsection}$ Department of Computing Science, Ume\aa{} University,  901 87 Ume\aa, Sweden.}
\address{${}^{\dagger}$ Institute for Applied and Numerical Mathematics, Karlsruhe Institute of Technology, Englerstr.~2, 76131 Karlsruhe, Germany.}
\email{roland.maier@kit.edu}
\address{${}^{\invdag}$ Department of Mathematics, Saarland University, Campus E 1.1, 66123 Saarbrücken, Germany.}
\email{felix.krumbiegel@uni-saarland.de}
\date{\today}
%
%
\begin{abstract}
  In this paper, we introduce a higher-order multiscale method for time-dependent problems with highly oscillatory coefficients. Building on the localized orthogonal decomposition (LOD) framework, we construct enriched correction operators to enrich the multiscale spaces, ensuring higher-order convergence without requiring assumptions on the coefficient beyond boundedness. This approach addresses the challenge of a reduction of convergence rates when applying higher-order LOD methods to time-dependent problems. Addressing a parabolic equation as a model problem, we prove the exponential decay of these enriched corrections and establish rigorous a priori error estimates. Numerical experiments confirm our theoretical results. 
\end{abstract}

\maketitle




\section{Introduction}
We consider the numerical solution of parabolic partial differential equations with highly oscillatory diffusion coefficients. Such problems arise in the context of, e.g., thermal conductivity in composite materials and heat conduction in batteries. It is well known that a standard spatial discretization requires the mesh size $h$ to be smaller than the oscillation scale $\varepsilon$ of the diffusion coefficient to get a reasonable approximation in the first place. When $\varepsilon \ll 1$ relative to the domain size, the requirement~$h < \varepsilon$ leads to prohibitively high computational cost, especially for time-dependent problems. In such cases, multiscale methods provide a more efficient and suitable alternative. 

Over the years, many multiscale methods have been developed to address the challenge of rough coefficients. 
In the context of parabolic problems, we exemplarily refer to~\cite{Minz07,OwhZ07,AbdV12,Tan2019,Owhadi2017,CHUNG2018419, SchSme,Eckhardt2023, LjuMaiMal,HuLC25}. 
In this work, we consider the localized orthogonal decomposition (LOD) method, which builds upon ideas of the variational multiscale method \cite{Hughes1998}. The LOD method is based on a decomposition of fine-scale finite element solution space into a low dimensional multiscale space and a high-dimensional remainder space. The multiscale space is spanned by basis functions containing information on the coefficient and have local support in small subdomains on a coarse mesh. The LOD method was originally proposed for elliptic problems~\cite{MalP14} based on first-order continuous finite elements, and then successfully applied to solve time-dependent problems, e.g., parabolic-type problems \cite{Malqvist2018,MalP17,AltCMPP20} and hyperbolic problems~\cite{Abdulle2017,PetS17,MaiP19,GeeM23}, see also the review paper~\cite{Altmann2021} and the textbook~\cite{MalP20}. 

Higher-order LOD methods are proposed in \cite{Mai21} for elliptic problems and improvements and generalizations regarding localization strategies are presented in~\cite{DonHM23,HauLM25}. These methods may use a nonconforming space consisting of Legendre polynomials of partial degree~$\leq p$ on a coarse mesh to construct a conforming multiscale space. More precisely, the basis functions of the multiscale space are solutions to a constrained minimization problem, where the constraints enforce that the $L^2$-projection of each basis function into the nonconforming space is a unique Legendre polynomial of degree~$\leq p$ on one element. The global multiscale basis functions exhibit exponential decay, allowing to localize the computation to small subdomains. Rigorous a priori error estimates are derived, showing optimal $p+2$ convergence rate in the $H^1$-norm for problems with general~$L^{\infty}$-coefficients and sufficiently regular right-hand side functions. 

In \cite{KruM25}, the higher-order LOD method was proposed for solving the wave equation with highly oscillatory wave speed in space. 
{From classical theory it is well-known }that higher-order convergence rates require higher regularity in space of {the }solution and its time derivatives. However, not much more than $H^1$-regularity in space can be expected with general $L^{\infty}$-coefficients. Consequently, arbitrarily high orders of convergence cannot be obtained. We note that the same argument is valid for parabolic PDEs. 

{In this paper, we consider a parabolic model problem, where the diffusion coefficient is time-independent but highly varying in space and develop a novel enriched higher-order LOD method. The main contribution of the proposed method is to address the order-reduction phenomenon observed in \cite{KruM25} for time-dependent problems. This is achieved by introducing appropriate enrichments of the higher-order multiscale space from the elliptic setting in~\cite{DonHM23}. Through this enrichment, we recover and prove the optimal higher-order spatial convergence rates in the very general setting of coefficients that are merely bounded. Our approach does not require additional spatial regularity of the solution or its time derivatives, but utilizes higher temporal regularity of the solution commonly satisfied for sufficiently regular initial data. In this way, we circumvent the regularity limitations encountered in earlier works. The construction of the enrichments itself is nontrivial and very technical, because they need to capture the additional fine-scale information while remaining efficient to compute. Importantly, we prove an exponential decay property of the optimal enrichments, ensuring that they can be localized. In particular, it is possible to compute them on the same localized patches as the standard higher-order LOD basis functions. While we focus on a parabolic model problem in this paper, the proposed enrichments can generally be transferred to other linear time-dependent PDEs such as, e.g., the wave equation, although the corresponding error analysis needs to be suitably adjusted. }

The rest of the paper is organized as follows. In~\cref{sec:num_hom}, we present the model problem and review the higher-order LOD method. After that, we introduce the enriched correction operators proposed in this paper, and state the localization strategy and semi-discrete a priori error estimate with optimal convergence rate. In~\cref{sec:proofs}, we give the complete proofs of the main theorems. A fully discrete method is given in~\cref{sec:numerical_examples} and we present some numerical examples to verify the theoretical analysis. We draw conclusion in~\cref{sec:conclusion}. 

\subsection*{Notation} We write~$a\lesssim b$ if~$a\leq Cb$, where the constant~$C$ is independent of the mesh {sizes~$H$ and $h$}, the localization parameters~$\ell$ and~$\lambda$, and the fine-scale parameter~$\varepsilon$. The constant may depend on the domain~$\Omega$, the regularity of the mesh~$\mathcal{T}_H$, the polynomial degree~$p$, the dimension~$d$, and the final time~$T$. We denote~$a\sim b$ if~$a\lesssim b \lesssim a$.


\section{Numerical Homogenization}\label{sec:num_hom}
In this section, we present the higher-order extension to the localized orthogonal decomposition ($p$-LOD) method, and also introduce a new technique which overcomes the order reduction shown in \cite{KruM25}. We state a remarkable localization result and the convergence rate for the error in the semi-discrete setting at the end of the section and prove them in~\cref{sec:proofs}. 

\subsection{Model problem}
We consider the parabolic partial differential equation (PDE)
\begin{equation}\label{eq:heat_model_problem}
  \begin{aligned}
    \diff t u - \divergence (A \nabla u) &= f \qquad &&\textrm{in~} \Omega\times [0, T),\\
    u &= 0 \qquad &&\textrm{on~} \Gamma\times [0, T),\\
    u &= u_0 \qquad &&\textrm{in~} \Omega\times\{0\},
  \end{aligned}
\end{equation}
where $\Omega\subset\R^d$ is an open and bounded Lipschitz domain, $\Gamma=\partial \Omega$ denotes the (Dirichlet) boundary of $\Omega$, and $T>0$. {Further, for simplicity we assume that~$\Omega$ may be decomposed into $d$-rectangles. }We have in mind coefficients $A$ that are highly oscillatory on a fine scale $0<\varepsilon\ll 1$, and be as general as $A\in \Space{L}{\Omega}[\infty]$, with positive lower and upper bounds $0<\alpha\leq A(x)\leq\beta<\infty$ for {almost every}~$x\in\Omega$. In our current setting, $A$ is assumed to be constant in time. In general, matrix-valued coefficients may also be considered, but for simplicity we restrict ourselves to scalar coefficients. 
For the discretization we introduce the weak formulation of the parabolic PDE~\eqref{eq:heat_model_problem}, which seeks a function $u$ that solves 
\begin{equation}\label{eq:heat_var_form}
    \langle \diff t u,v\rangle_{H^{-1}(\Omega)\times\Hloc} + a\IP{u}{v}[] = \langle f,v\rangle_{H^{-1}(\Omega)\times\Hloc},
\end{equation}
for all $v\in \Hloc$ {and almost every~$t\in[0,T)$}, where $a\IP{u}{v}[] = \IP{A\nabla u}{\nabla v}$ {denotes the $a$-induced inner product}, and~$\langle u,v \rangle_{H^{-1}(\Omega)\times\Hloc}$ {denotes} the dual brackets, with initial condition $u(\cdot,0)=u_0$. The following assumptions assure that a solution to \eqref{eq:heat_var_form} exists, with certain regularity in space and time. First, let~{$\{\mathcal{T}_H\}_{H>0}$} be a family of regular decompositions of $\Omega$ into quasi-uniform $d$-rectangles for a coarse mesh size $H>0$, cf.~\cite{Cia78}. {Note that more general meshes composed of, e.g., simplices or $d$-parallelograms could be considered as well.} Then, we define the space
\begin{equation*}
    \Space{H}{\mathcal{T}_H}[k] = \{v\in\LL\mid v\vert_K \in \Space{H}{K}[k]~\mathrm{for~all}~K\in\mathcal{T}_H\}, \quad {k \geq 1}, \quad \Space{H}{\mathcal{T}_H}[0] \coloneqq L^2(\Omega). 
\end{equation*}
 
\begin{assumption}[Regularity]\label{ass:regularity}
    Let $k,m\in\N_0$. Consider the following well-preparedness and compatibility assumptions: 
    \begin{enumerate}[label={(A\arabic*)}]\setcounter{enumi}{-1}
        \item Let~$f\in\Space{C}{{[0,T];\Space{H}{\mathcal{T}_H}[k]}}[m]$,
        \item let $u(\cdot,0) = u_0\in\Hloc$, 
        \item \label{ass:compatibility} let $\diff t[\nu ]u(\cdot,0) = \diff t[\nu -1] f(\cdot,0) - \divergence (A \nabla \diff t[\nu -1]u(\cdot,0))\in\Hloc$, for $\nu =1, \dots, m$,
        \item there exists a constant $C_{\mathrm{init}}$, which is independent of $\varepsilon$, such that
        \begin{equation*}
            \sum_{\nu =0}^{m}\Norm{\diff t[\nu ] u(\cdot,0)}[{\Space{H}{\Omega}[1]}] \lesssim C_{\mathrm{init}}.
        \end{equation*}
    \end{enumerate}
\end{assumption}
Assumption \ref{ass:regularity} ensures that a unique weak solution $u$ to the weak formulation~\eqref{eq:heat_var_form} exists, and has the regularity 
\begin{equation*}
    u\in\Space{C}{{[0,T];\Hloc}}[m]\cap\Space{C}{{[0,T];\LL}}[m+1].
\end{equation*}
Furthermore, there exists a constant $C_{\mathrm{data}}$ independent of $\varepsilon$ such that 
\begin{multline}\label{eq:Cdata}
    \sum_{\nu=0}^{m}\Norm{u}[{\Space{C}{{[0,T];\Space{H}{\Omega}[1]}}[\nu]}] + \Norm{u}[{\Space{C}{{[0,T];\LL}}[m+1]}]\\
    \lesssim \sum_{\nu=0}^{m} \Norm{f}[{\Space{C}{{[0,T];\Space{H}{\mathcal{T}_H}[k]}}[\nu]}] + \sum_{\nu=0}^{m}\Norm{\diff t[\nu] u(\cdot,0)}[{\Space{H}{\Omega}[1]}]\lesssim C_{\mathrm{data}}. 
\end{multline}
For a proof, we refer to {\cite[Thm.~7.1.5]{Eva10}}. 

\begin{remark}[Regularity]
    We note that~$k=0$ { and~$m=0$} is sufficient for the existence of a solution{. I}ncreasing the spatial regularity of $f$ (and also of $u_0$) does in general not increase the spatial regularity of $u$. This  follows from the assumption $A\in L^\infty(\Omega)$. {However, the lowest regularity~$k=0$ is not interesting for this work as it would only allow for lowest-order convergence. }%
    {That is, t}he parameter $k$ is still crucial to show the higher-order convergence of our method. In the following we make further assumptions to simplify the presentation. We assume~$k\geq 1$ and~$p=k-1$, such that the spatial regularity does not restrict the final convergence rate~$r=p+2=k+1$ in space. %
    {Regarding the temporal regularity, $m=0$ is the minimum that is required for existence and uniqueness of a solution. However, classical time stepping methods, such as the backward differentiation formulas, require higher temporal regularity for higher-order convergence in time. Additionally, for the proposed eho-LOD method presented later, we will trade missing spatial regularity for temporal regularity, and we state the precise requirement on the temporal regularity~$m$ in \Cref{thm:semi_discrete}.} 
    
    The well-preparedness and compatibility assumptions can be met, e.g., for $\nu=0,\dots,m-1$ if~$\diff t[\nu]u(\cdot,0)\equiv 0$, and $\diff t[\nu] f(\cdot,0)\equiv 0$. A possible interpretation of compatibility assumptions~\ref{ass:compatibility} could be heat conducting through an oscillatory medium~$A$ in space, where the solution naturally has to fulfill the PDE at the initial time. An interpretation for the case where all initial states are zero suggests that a zero initial state is externally excited by a source~$f$. %
    {Note that if the well-preparedness assumptions are not met, reduced convergence may also be observed in practice.}
\end{remark}

\subsection{Problem-adapted corrections}\label{subsec:corrections}
In this section, we introduce the basics of the stabilized $p$-LOD method following~\cite{DonHM23}, see also~\cite{Mai21}. The general idea lies in adapting coarse scale basis functions in such a way that when applied to the parabolic model problem~\eqref{eq:heat_var_form} with the specific coefficient $A$ they have much better approximation properties than classical finite element basis functions. 
In this section, we will show that these bases are not suitable for parabolic problems and we provide a construction of enriched corrections that will converge in the optimal rate. We will supply a localization strategy for the enriched corrections and give a convergence result that is independent of the spatial regularity of the solution. 

Let $p\in\N_0$ and $V_H$ denote the space of piecewise polynomials (with respect to the mesh) up to partial degree $p$, more precisely
\begin{equation*}
  V_H\coloneqq \{v\in\Space{L}{\Omega}[2] \mid v\vert_K \textrm{~is a polynomial of partial deg.}\leq p\textrm{~for~an~element~}K\in \cal T_H\}.
\end{equation*}
Further, let~$\LLProjection$ be the $\LL$-projection onto $V_H$. We have for~$\kappa\in\N_0$ with $\kappa\leq p+1$ 
\begin{equation}\label{eq:Pi_estimate}
    \Norm{(1-\LLProjection)v} \lesssim H^{\kappa}\Norm{v}[{\Space{H}{\mathcal{T}_H}[\kappa]}],
\end{equation}
for any $v\in\Space{H}{\mathcal{T}_H}[\kappa]$, see,~e.g.,~\cite{HouSS02}. 
For the remainder of this work, we denote the restriction of any space~$V$ to functions with support in a subdomain~$S\subset \Omega$ by~$V(S)$.

First, we construct the stabilized~$p$-LOD basis functions following~\cite{DonHM23}. Let~$\mathfrak{L}\coloneqq \bigcup_{K\in\mathcal{T}_H}\{\Lambda_{K,i}\}_{i=1}^{(p+1)^d}$ a basis of $V_H${, where on each element~$K\in\mathcal{T}_H$, the function~$\Lambda_{K,i}$ is a tensor-product Legendre polynomial mapped from the reference cube~$[-1, 1]^d$ to~$K$ and extended by zero outside~$K$.} %
By~\cite[Cor.~3.6]{Mai21}, there exists to each Legendre polynomial~$\Lambda_{K, i}\in\LegendreBasis$ a so-called \emph{bubble function}~$b_{K, i}\in\Space{H}{K}[1][0]$ that shares the same (local)~$\Space{L}{K}[2]$-projection, i.e.,
\begin{equation*}
  \IP{b_{K, i}}{v_H}[{\Space{L}{K}[2]}]=\IP{\Lambda_{K, i}}{v_H}[{\Space{L}{K}[2]}],
\end{equation*}
for all $v_H\in V_H(K)$. We note that these bubble functions can be computed (see~\cite[Rem.~7.1]{DonHM23}). 
Analogously to~\cite{DonHM23}, we can define a linear operator $\bubbleOperator\colon V_H\to B_H$ {with the bubble space} $B_H=\spann \bigcup_{K\in\cal T_H}\{b_{K, i}\}_{i=1}^{(p+1)^d}$, that maps each basis function $\Lambda_{K, i}\in\LegendreBasis$ to its corresponding bubble function $b_{K, i}$. 
Note, that this operator can be extended to $\LL$ by concatenation $\bubbleOperator v\coloneqq \bubbleOperator\LLProjection v$ for all~$v\in\LL$. 
From~\cite[eq.~(3.5)]{DonHM23} we have for any~$v\in L^2(K)$ and~$K\in\mathcal{T}_H$
\begin{equation}\label{eq:stability_H1_bubble_op}
    \Norm{\nabla\bubbleOperator v}[L^2(K)] \lesssim H^{-1}\Norm{v}[L^2(K)].
\end{equation}
{While the operator~$\bubbleOperator$ introduces a favorable conformity, it still leads to an unsatisfactory behavior of the corresponding multiscale construction; see \Cref{rem:bubbles} below and~\cite{Mai21} for further details.} %
{Thus,} a better suited operator~$\xInterpolation$ is required. The definition of this operator is based on~\cite{Altmann2021, HauP22} and more involved. {We define a quasi-interpolation operator~$\quasiInterpolation$ that maps onto continuous piecewise (multi-)linear polynomials with zero boundary. That is, for any~$v\in\Hloc$ and an interior node~$z\in\mathcal{T}_H$ we have %
\begin{equation*}
    (\quasiInterpolation v ) (z) = \sum_{K \in \mathcal{T}_H}\frac{|K|}{|\omega_z|}  \int\limits_{K} v\d x,
\end{equation*}
where~$\omega_z = \{K\mid \overline{K}\cap z\neq\emptyset\}$. %
The quasi-interpolation can be interpreted as a concatenation of the $\LL$-projection onto the space of piecewise constants and an averaging operator that maps onto continuous piecewise affine functions with zero trace at the boundary. %
For further details we refer to \cite{DonHM23,KruM25}. }%
Finally, we can define an~$\Hloc$-conforming subspace that is adapted to the problem at hand. Define the \emph{stabilized bubble space} $U_H=\xInterpolation V_H$, where the operator~$\xInterpolation$ (\cite[eq.~(3.6)]{DonHM23}) is given for any~$v\in\LL$ by 
\begin{equation*}
    \xInterpolation v= \quasiInterpolation v + \bubbleOperator (v-\quasiInterpolation v).
\end{equation*}
{Let now $\ell\in\N$, and define the \emph{patch of order $\ell$} around a subdomain $S\subset\Omega$ by 
\begin{equation*}
    \mathtt{N}^\ell(S)=\mathtt{N}^1(\mathtt{N}^{\ell-1}(S)),\quad\ell\geq 2,\qquad \mathtt{N}^1(S)= \mathtt{N}(S)=\bigcup\{\Bar{K}\in\mathcal{T}_H\mid \Bar{S}\cap\Bar{K}\neq\emptyset\}. 
\end{equation*}
Then, we have stability for the operator~$\xInterpolation$ in the following sense. }
\begin{lemma}[{\cite[Lem.~3.2]{DonHM23}}]
    The operator $\xInterpolation$ is a projection onto $U_H$ with the same kernel as the $L^2$-projection $\LLProjection$. Further, we have for any~$v\in\Hloc$ and~$K\in\mathcal{T}_H$ 
    \begin{equation*}
      \Norm{\nabla\xInterpolation v}[{\Space{L}{K}[2]}] + H^{-1}\Norm{(1-\xInterpolation) v}[{\Space{L}{K}[2]}]\lesssim \Norm{\nabla v}[{\Space{L}{\mathtt{N}(K)}[2]}].
    \end{equation*}
\end{lemma}%
{\begin{remark}[Bubble functions]\label{rem:bubbles}
    As the construction is quite involved, we give a short explanation on how the spaces are set up. In~\cite{Mai21}, the higher-order LOD method was introduced, where the construction of the LOD basis functions follows a constraint minimization problem. Implicitly, the construction uses higher-order polynomials and finds suitable bubble functions with equal~$L^2$-projections and $H^1_0$-conformity. This is achieved with the operator~$\bubbleOperator$. %
    The issue with this construction is that $\bubbleOperator$ is not an appropriate coarse operator to provide a starting point for an suitable multiscale construction. In particular, replacing piecewise constant functions by local $H^1_0$-functions turns out to be problematic. The extended bubble operator~$\xInterpolation$ works around this issue by allowing a slightly enlarged support to enforce conformity. %
    We refer to~\cite{Altmann2021,KruM25,DonHM23} for visualizations of the extended bubble functions. %
\end{remark}}

{As a next step, we want to `correct' the functions in the extended bubble space by suitable function in }
$W=\ker \LLProjection = \ker \xInterpolation$. For any~$K\in\mathcal{T}_H$, and~$\ell\in\N$, let the \emph{localized element-wise correction operator}~$\Corrector_K^{[\ell]}\colon \Hloc\to W$ be defined for any~$v\in\Hloc$ by
\begin{equation}\label{eq:CK_definition}
    \rIP{\Corrector_K^{[\ell]} v}{w}{\NbK}=\rIP{v}{w}{K},
\end{equation}
for all $w\in W(\NbK)$, where $\rIP{u}{v}{K} = \IP{A\nabla u}{\nabla v}[L^2(K)]$. The \emph{localized correction operator} is given by $\Corrector^{[\ell]} = \sum_{K\in\cal T_H}\Corrector_K^{[\ell]}$. The \emph{localized multiscale space} is defined as $\tV^{[\ell]}=(1-\Corrector^{[\ell]})\,U_H$. 
\begin{remark}[global corrections]\label{rem:global_corrections}
    Choosing $\ell\geq\tfrac{\diam(\Omega)}{H}$, we have~$\NbK=\Omega$ for all~$K\in\mathcal{T}_H$. In this setting we formally set $\ell=\infty$ and define the \emph{correction operator}~$\Corrector \coloneqq \Corrector^{[\infty]}$, the \emph{element-wise correction operators}~$\Corrector_K \coloneqq \Corrector_K^{[\infty]}$, and the \emph{multiscale space}~$\widetilde{V}_H= \widetilde{V}_H^{[\ell]}$. 
\end{remark}
{\begin{remark}[Correcting bubble functions]
    The procedure of constructing the multiscale space is as follows. We start with a coarse polynomial space, and seek bubble functions that are~$\Hloc$-conforming and whose $L^2$-projection onto the polynomial space equal the basis functions of the polynomial space. The correction then adjusts the bubble functions such that they minimize the elliptic energy while preserving the~$L^2$-projection onto polynomials. %
    For elliptic problems, this construction leads to favorable orthogonality properties of the Galerkin error, which can be used to extract optimal orders of convergence from the right-hand side. 
\end{remark}}

For the (element-wise) correction operator the following results hold. 
\begin{lemma}[{\cite[Lem.~5.1]{DonHM23}}]\label{lem:CK_decay}
    Let~$K\in\mathcal{T}_H$, and~$\ell\in\N$. Then 
    for any $v\in\Hloc$ 
    \begin{equation*}
        \Norm{\nabla \Corrector_K v}[L^2(\Omega\setminus\NbK)] \lesssim \exp^{-C\ell}\Norm{\nabla\Corrector_Kv}. 
    \end{equation*}
\end{lemma}
\begin{lemma}[{\cite[Lem.~A.1]{DonHM23}}]\label{lem:CK_decay_error}
    Let~$K\in\mathcal{T}_H$, and~$\ell\in\N$. Then 
    for any $v\in\Hloc$ 
    \begin{equation*}
        \Norm{\nabla (\Corrector_K^{[\ell]} - \Corrector_K) v} \lesssim \exp^{-C\ell}\Norm{\nabla\Corrector_Kv}. 
    \end{equation*}
\end{lemma}
\begin{lemma}[{\cite[Thm.~5.2]{DonHM23}}]\label{lem:C_decay_error}
    Let~$\ell\in\N$. Then for any $v\in\Hloc$ 
    \begin{equation*}
        \Norm{\nabla (\Corrector^{[\ell]} - \Corrector) v} \lesssim \exp^{-C\ell}\Norm{\nabla v}. 
    \end{equation*}
\end{lemma}

The results in~\cite{KruM25} (in the context of the wave equation) show that the spaces $\tV^{[\ell]}$ are not well suited for time-dependent cases, which also holds similarly for parabolic equations. To illustrate the issue, let Assumption \ref{ass:regularity} hold, and we choose $p=k-1$. 
The reduced order can be observed when considering the projection error, that is, the error between the exact solution and its orthogonal projection into the multiscale space~$\tV^{[\ell]}$. Let the projection into the (localized) multiscale space $\tProjection^{[\ell]}\colon\Hloc\to\tV^{[\ell]}$ be given by
\begin{equation}\label{eq:ms_proj_def}
    \HIP{\tProjection^{[\ell]} v}{\tH v} = \HIP{v}{\tH v},
\end{equation}
for all $\tH v\in\tV^{[\ell]}$, and~$\tProjection=\tProjection^{[\infty]}$. For the sake of readability, we omit the arguments and note that the following estimate holds for all times~$t\in[0,T]$. Then, we have for the projection error $u-\tProjection u$, where $u$ denotes the weak solution of \eqref{eq:heat_var_form} 
\begin{multline}\label{eq:tprojection_error}
    \Norm{\nabla(u-\tProjection u)}^2 \lesssim \HIP{u-\tProjection u}{u-\tProjection u} = \IP{f}{u-\tProjection u} - \IP{\diff t u}{u-\tProjection u}\\
    \begin{aligned}
        &\lesssim H^{p+2}\Norm{f}[{\Space{H}{\mathcal{T}_H}[k]}]\,\Norm{\nabla(u-\tProjection u)} + H^2\Norm{\nabla \diff t u}\, \Norm{\nabla(u-\tProjection u)},
    \end{aligned}
\end{multline}
where in the equality we use Galerkin orthogonality, and in the second inequality we employ~$u-\tProjection u\in\kernel \LLProjection$, the Cauchy-Schwarz inequality, and \eqref{eq:Pi_estimate}. Here, we observe that the optimal rate~$r=p+2$ can only be observed if $k=1$, and is otherwise capped (independently of~$k$ and~$p$) at $r=2$. For more details we refer to~\cite[Lem.~3.13]{KruM25}. 
In equation~\eqref{eq:tprojection_error} we observe that the $p$-LOD method deals very well with the elliptic operator of the PDE but {the argument} cannot be extended to the time-dependent setting. We note that for sufficiently smooth coefficients~$A$ and initial data it is possible to increase the spatial regularity of $\diff t u$ and thus to obtain higher-order rates. 

The goal is to introduce so-called \emph{enriched corrections} that deal with the reduced order without assuming higher regularity than~$A\in L^\infty(\Omega)$. 
For any $v\in \LL$, define the \emph{enriched correction operator}~$\addCorrector v\in W$ as the solution to 
\begin{equation}\label{eq:tildeC_def}
  \HIP{\addCorrector v}{w} = -\IP{v}{w},
\end{equation}
for all $w\in W$. 
{\begin{remark}
    The  similarity of the enriched correction operator with the classical correction operator  comes from the design. As the classical correction only deals with the elliptic operator, the enriched correction operator is used in the following to `correct' also the problematic~$L^2$-term with the time derivative of~$u$. %
    Here, we first consider an ideal setting, where the enriched operator maps into an infinite-dimensional globally defined space. Moreover, we observe that the operator perfectly circumvents reduced convergence rates. 
    
    We also note the similarity between~\eqref{eq:tildeC_def} and the right-hand side corrections introduced in~\cite{Hellman2017}, though the purpose is completely different. 
\end{remark}}

\begin{lemma}[Error of the ideal enriched correction]\label{lem:error_addCorr}
    Consider Assumption \ref{ass:regularity}. If we choose $p=k-1$ and~$\ell=\infty$, then the projection of $u$ into the multiscale space combined with the enriched correction operator $\tCorrector{u}$ converges optimally,~i.e., 
    \begin{equation}\label{eq:rho_estimate}
      \begin{aligned}
        \sup_{t\in[0,T]} \Norm{\nabla (u(\cdot,t) - (\tProjection u(\cdot,t) + \tCorrector{u(\cdot,t)}))} \lesssim H^{p+2} \Norm{f}[{\Space{C}{[0,T];\Space{H}{\mathcal{T}_H}[k]}[]}].
      \end{aligned}
    \end{equation}
\end{lemma}

\begin{proof}
    We omit the argument~{$t\in[0,T]$} for readability and note that the estimate holds for all~$t\in[0,T]$. Let~$u - (\tProjection u + \tCorrector{u}) \eqqcolon \psi\in W$. Then $\psi\in\kernel \LLProjection$ and by the definition of~$\tProjection$ and~$\addCorrector$, we obtain with~\eqref{eq:Pi_estimate} 
    \begin{equation*}
      \begin{aligned}
        \HIP{u - (\tProjection u + \tCorrector{u})}{\psi} &= \IP{f}{\psi} - \IP{\diff t u}{\psi} + \IP{\diff t u}{\psi} \\
        & \lesssim H^{p+2} \Norm{f}[{\Space{H}{\mathcal{T}_H}[k]}] \, \Norm{\nabla\psi}.
      \end{aligned}
    \end{equation*}
    Now~\eqref{eq:rho_estimate} follows from $\Norm{\nabla \psi}^2\lesssim a\IP{\psi}{\psi}[]$ and taking the supremum over all~$t$. 
\end{proof}
\begin{remark}
    We note here that the proof can be applied in similar fashion to the wave equation, if the enriched correction operator~$\addCorrector$ is applied to the second time derivative of the solution~$u$.
\end{remark}

Lemma~\ref{lem:error_addCorr} indicates that the enriched correction operator can recover the optimal order~$r=p+2$. However, using this operator is practically unfeasible. This stems on the one hand from the fact that we have no access to the function~$\diff t u$. We can work around that problem by replacing~$\diff t u$ for instance with its projection into the multiscale space~$\tProjection (\diff t u)$, and thus obtaining an approximation 
\begin{equation}\label{eq:solution_approximation_expansion_j1}
    u\approx\tProjection u + \addCorrector(\diff t \tProjection u)
\end{equation}
Now we can plug this approximation for~$u$ into the second term of~\eqref{eq:solution_approximation_expansion_j1} again. This can be carried out recursively to obtain 
\begin{equation}\label{eq:solution_approximation_expansion}
    u \approx \tProjection u + \addCorrector(\diff t \tProjection u) + \addCorrector^2(\diff t[2] \tProjection u) + \dots + \addCorrector^\nu(\diff t[\nu] \tProjection u). 
\end{equation}
On the other hand, the enriched correction operator is obtained by solving a global problem such that a good localization strategy is required that produces minimal overhead compared to the localization of the correction operator~$\Corrector$. 

{For the localization of the enriched correction operator, we present below a practical version that is used to construct the enriched multiscale spaces. In order to prove the desired localization result, the construction is more involved and is postponed to \Cref{subsec:proof_loc}. Let $\Lambda_{K,i}\in V_H$ be a basis function, then the multiscale basis function is given by $\tilde{\Lambda}_{K,i}^{[\ell]} = (1-\Corrector^{[\ell]}) \xInterpolation\Lambda_{K,i}\in W(\NbK)$, that is supported on the~$\ell$-patch around~$K$. %
The enriched basis functions are now given as the solution to the enriched corrector problem~\eqref{eq:tildeC_def} restricted to the respective patches. That is, we define the enriched basis function~$\check{\Lambda}_{K,i}^{\nu,\mathrm{loc}} \in W(\Nb)$ for~$\nu=1,\dots,j$ recursively by %
\begin{equation*}
\begin{aligned}
    \rIP{\check{\Lambda}_{K,i}^{1,\mathrm{loc}}}{w}{\NbK} &= -\IP{\tilde{\Lambda}_{K,i}^{[\ell]}}{w}[L^2(\NbK)],\\
    \rIP{\check{\Lambda}_{K,i}^{\nu,\mathrm{loc}}}{w}{\NbK} &= -\IP{\check{\Lambda}_{K,i}^{\nu-1,\mathrm{loc}}}{w}[L^2(\NbK)],\qquad \nu=2,\dots,j,
\end{aligned}
\end{equation*}
for all~$w\in W(\Nb)$. It is straight-forward to see that this construction yields functions that are all supported on patches of equal size. The justification for this construction will be given in \Cref{subsec:proof_loc}. 

We note that with the standard approach to derive the localization error (as for the classical correction operator), only a suboptimal result can be achieved. 
That is, for the classical correction operator, the right-hand side is supported locally on one element~$K$, and the standard approach to derive localization error estimates yields an exponential decay property away this element. For the enriched basis function, however, we would need to consider a patch around the patch~$\Nb$, resulting $\Nb[K][2\ell]$. Overall, the size of the patch would increase with the recursion in $\nu$. Instead, in \Cref{subsec:proof_loc} we make use of the fact that the basis function on the right-hand side already has a decay, and derive a sharper localization estimate. %

Eventually, the enrichment space is now defined as the span of the computed basis functions, i.e., %
\begin{equation}\label{eq:loc_enriched_basis}
    {\widehat{W}^{j,\mathrm{loc}}_H} = \sum_{\nu=1}^{j}\,\spann \bigcup_{K\in\mathcal{T}_H}\{\check{\Lambda}_{K,i}^{\nu,\mathrm{loc}}\}_{i=1}^{(p+1)^d}.
\end{equation}
The final enriched multiscale space is then given by the direct sum %
\begin{equation*}
    \cVloc=\tVloc + \hWloc.
\end{equation*}
That is, we enrich the multiscale space with localized functions in the kernel of the~$L^2$-projection that are able to capture relevant features introduced by the time-dependence of the  parabolic equation. }%
{The creation of these spaces introduces a computational overhead which scales with the  polynomial degree $p$, the number of layers~$\ell$, the number of enriched corrections~$j$ and the dimension~$d$. However, due to an overall reduction of the dimension and parallelization, the eho-LOD spaces provide a well-chosen spatial discretization that outperforms classical methods after a certain amount of time steps. }

\subsection{{Semi-discrete} enriched higher-order LOD method}\label{subsec:eholod_method}
{In the following, we introduce our method regarding a discretization in space only. This makes the analysis more readable and keeps the notation to a minimum. We emphasize, however, that an additional temporal error analysis can be derived with classical arguments.} We use $\cVloc$ as trial and test space for the weak formulation \eqref{eq:heat_var_form} to obtain the {\emph{semi-discrete enriched higher-order (eho-)LOD method}}: Seek $\check{u}_H\colon [0,T] \to \cVloc$ such that 
\begin{equation}\label{eq:practical_method}
    \IP{\diff t \check{u}_H(t)}{\check{v}_H}+\HIP{\check{u}_H(t)}{\check{v}_H} = \IP{f(t,\cdot)}{\check{v}_H},
\end{equation}
for all $\check{v}_H\in\cVloc$ and all~$t\in[0,T]$, with the initial condition 
\begin{equation}\label{eq:sd_init_condition_practical}
    \cu(0) = \tProjectionloc u_0 + \hProjectionloc u_0,
\end{equation}
where the projection into the localized multiscale space is defined in~\eqref{eq:ms_proj_def}. Further, 
\begin{equation}\label{eq:add_proj_def}
  \hProjectionloc u = \sum_{\nu=1}^j (\addCorrector^{\mathrm{loc}})^{\nu}(\diff t[\nu] \tProjectionloc u)\quad\text{and}\quad  \hProjection u = \sum_{\nu=1}^j \addCorrector^{\nu}(\diff t[\nu] \tProjection u).
\end{equation}
Finally we can state the main error estimate. 
\begin{theorem}[Error of the semi-discrete solution]\label{thm:semi_discrete}
  Let Assumption~\ref{ass:regularity} hold with $k\geq 1$ and $m=\lceil \frac{k-1}{2}\rceil + 2$. Choose the optimal parameters,~i.e., the polynomial degree~$p=k-1$, the number of enriched corrections~$j=\lceil \tfrac{p}{2}\rceil$, the localization parameter~$\ell\sim C_p|\log H|$, and the localization parameters for the enriched corrections~$\lambda_G$ as in~\Cref{thm:localization}. Further, let $u$ be the solution to \eqref{eq:heat_var_form}, and let $\check{u}_H\in C^m([0,T];\cVloc)$ solve \eqref{eq:practical_method}. Then, the error~$e(t)=u(\cdot,t)-\check{u}_H(t)$ can be bounded independently of $\varepsilon$ by 
  \begin{equation*}
    \sup_{t\in[0,T]}\big[\Norm{e(t)} + \Norm{\nabla e(t)}\big] \lesssim_{T,p} H^{p+2}\, C_{\mathrm{data}}, 
  \end{equation*}
  with $C_\mathrm{data}$ from \eqref{eq:Cdata}. 
\end{theorem}
\begin{remark}
    In this work, we do not track the dependence on the polynomial degree. {Such dependencies are studied in~\cite{Mai21}. We emphasize, however, that the (hidden) constants in the error estimate of~\Cref{thm:semi_discrete} have an advantageous scaling with respect to the polynomial degree~$p$}, which can also be observed in the numerical experiments in~\cref{sec:numerical_examples}. %
    {The constant~$C_p$ theoretically scales like~$p^2$ as derived in~\cite{Mai21}. However, numerical experiments clearly indicate that this scaling is pessimistic and a scaling as~$p^s$ for some $0 < s < 1$ seems more appropriate.}%
\end{remark}
\begin{remark}[Parameters]\label{rem:parameters}
    We shortly explain the choice of parameters in~\Cref{thm:semi_discrete}. It is possible to deal with right-hand sides that are only~$L^2$-regular,~i.e.,~$k=0$. In this case, the optimal choice (in terms of computational effort) is $p=j=0$, and we obtain a convergence rate of $r=1$, {and we require~$m=2$} (see equation~\eqref{eq:tprojection_error}). 
    In the proof, the restriction~$m\geq j+2$ will arise, which is always fulfilled with the choices in the theorem. 
    For $k=1$, we can also choose $p=j=0$, and {we have $m\geq 2$}. This case is excluded from the proof in~\cref{subsec:proof_main} but is rather standard, see~\Cref{rem:classical_holod}. 
\end{remark}

\begin{remark}[Practical implementation]\label{rem:practical_split}
    {In practice, the corrector problems, for both the classical corrector and the enriched corrector, need to be numerically solved to set up the spatial discretization. This is typically done by discretizing the respective equation with  a finite element space~$V_h$ with a mesh size $h$ that resolves the variations of the coefficient, i.e., $h\lesssim\varepsilon$. This results in an error that scales with the fine mesh size~$h$ and is thus small compared to the coarse error estimate of the (enriched) LOD method. The scale $h$ also needs to be fine enough to sufficiently capture polynomials on the mesh $\mathcal{T}_H$. That is, $H/h$ needs to be big enough.}%

    {Further, one has to consider the technical issue that the enriched corrections are of relatively small magnitude (scaled by about~$H^2$ per additional enrichment). This results in small eigenvalues for the system matrix and possibly badly conditioned systems. Here, we exploit that within each enrichment level~$\nu=0,\dots,j$ the norms are of comparable size. We use a Schur complement solver recursively for each of the levels. Thus, we only deal with better conditioned small matrices. In fact, we employ that it is possible to rewrite the eho-LOD method~\eqref{eq:practical_method} into a system of two coupled equations by splitting up the spaces,~i.e., we seek~$\tz\colon[0,T]\to\tVloc$ in the classical higher-order multiscale space and~$\hz\colon[0,T]\to\hWloc$ in the enriched multiscale space such that 
    \begin{equation*}
      \begin{alignedat}{3}
        &\IP{\diff t (\tz + \hz )}{\tv} && +\HIP{\tz + \hz }{\tv}&&=\IP{f}{\tv},\\
        &\IP{\diff t (\tz + \hz )}{\hv} && +\HIP{\tz + \hz}{\hv}&&=\IP{f}{\hv},
      \end{alignedat}
    \end{equation*}
    for all $\tv\in\tVloc$ and $\hv\in\hWloc$. This split can further be employed for each of the levels~$\nu=1,\dots,j$. That is, after a discretization in time the system is a block matrix system that may be solved using a Schur complement-type solver. }%
\end{remark}

\section{Proofs of the main theorems}\label{sec:proofs}

\subsection{Proofs of the localization results}\label{subsec:proof_loc}

{This section is devoted to proving exponential decay properties of the enriched correction operator. The decay will be proven similarly to the classical correction, i.e., we restrict the right-hand side of localized enriched correction~\eqref{eq:loc_enriched_basis} to element-wise contributions. However, in order to avoid growing patches with the number of enrichments, the enriched correction operator requires a novel localization strategy. 

For the localization of~$\addCorrector$, defined in~\eqref{eq:tildeC_def}, we first construct localized element-wise enriched correction operators analogously to the construction of the localized element-wise corrector~$\Corrector^{[\ell]}_K$. 
Let~$G\in\mathcal{T}_H$, and~$\lambda\in\N$. We define the \emph{localized element-wise enriched correction operator}~$\addCorrector_G^{[\lambda]}$ for~$v\in\LL$ as 
\begin{equation}\label{eq:loc_enriched_corrector_def}
\rIP{\addCorrector_G^{[\lambda]} v}{w}{\Nb[G][\lambda]}=-\IP{v}{w}[L^2(G)],
\end{equation}
for all $w\in W(\Nb[G][\lambda])$. 
{Up to this point, the definition of the localized element-wise enriched correction operator is analogous to the classical LOD method. For a sharper localization estimate, we allow $\lambda$ to depend on $G$, and write $\lambda_G$ instead. }%
For the definition of the \emph{localized enriched correction operator}~$\addCorrector^\mathrm{loc} = \sum_{G\in\cal T_H}\addCorrector_G^{[\lambda_G]}$, we choose distinct localization parameters~$\lambda_G\in\N$ for each element~$G\in\mathcal{T}_H$. This freedom will be crucial for proving the localization results later. 
Analogously to above, we can define the (global) element-wise enriched correction operators~$\addCorrector_G = \addCorrector_G^{[\infty]}$ for $\lambda_G=\infty$, and it follows that~$\addCorrector = \sum_{G\in\cal T_H}\addCorrector_G$. 

Due to the analogous construction of the element-wise enriched correction operators to the classical correction operators~$\Corrector_K^{[\ell]}$, the following lemmas can be proven employing the same arguments as in Lemma~\ref{lem:CK_decay} and Lemma~\ref{lem:CK_decay_error}, respectively. 
{\begin{corollary}\label{lem:DG_decay}
    Let $G\in\mathcal{T}_H$ and $\lambda_G\in\N$. Then we have for any $v\in\LL$ 
    \begin{equation*}
        \Norm{\nabla \addCorrector_G v}[L^2(\Omega\setminus\texttt{N}^{\lambda_G}(G))] \lesssim \exp^{-C\lambda_G}\Norm{\nabla\addCorrector_G v}.
    \end{equation*}
\end{corollary}%
\begin{corollary}\label{lem:DG_decay_error}
    Let $G\in\mathcal{T}_H$ and $\lambda_G\in\N$. Then we have for any $v\in\LL$ 
    \begin{equation*}
        \Norm{\nabla (\addCorrector_G^{[\lambda_G]} v - \addCorrector_G v)} \lesssim \exp^{-C\lambda_G}\Norm{\nabla\addCorrector_G v}.
    \end{equation*}
\end{corollary}}

Based on the intricate construction of the enriched basis functions, we can show the following localization error for the enriched basis functions. 

\begin{theorem}\label{thm:localization}
    Let~$\ell,\nu\in\N$ with~$H^2\ell^{d+1}\lesssim 1$. For any~$K\in\mathcal{T}_H$ and~$G_1\subset\mathtt{N}^\ell(K)$, choose~$\lambda_{G_1}(K)\in\N$ such that~$\lambda_{G_1}({K})=\ell-\mu_1$, where~$\mu_1 = \distance({G_1},K)$, and the distance~$\distance(\cdot,K)$ to the element~$K\in\mathcal{T}_H$ is defined by 
    \begin{equation*}
        \distance(S,K) = \mu,\qquad\textrm{if}~S\cap\mathtt{N}^\mu(K)\neq\emptyset\quad\mathrm{and}\quad S\cap\mathtt{N}^{\mu-1}(K)=\emptyset. 
    \end{equation*}
    Further, for any sequence of elements~$G_{i+1}\in\mathtt{N}^{\lambda_{G_{i}}(G_{i-1})}(G_i)$ for~$i=1,\dots,\nu$, where we formally set~$G_0=K$, choose~$\lambda_{G_{i+1}}(G_i)\in\N$ such that~$\lambda_{G_{i+1}}(G_i) = \lambda_{G_{i}}(G_{i-1}) - \mu_{i+1} $ with~$\mu_{i+1}=\distance(G_{i+1},G_i)$. 
    Then, for any function $v\in\Hloc$ we have 
    \begin{equation*}
        \Norm{\nabla ((\addCorrector^{\mathrm{loc}})^\nu - \addCorrector^\nu)(1-\Corrector^{[\ell]}) v} \lesssim \exp^{-C\ell}\Norm{\nabla v}. 
    \end{equation*}
\end{theorem}

\begin{figure}
  \centering
    \begin{tikzpicture} [scale=0.4]
        \draw[fill=black!25!white] (6,1) rectangle (13,8);
        \draw[fill=black!50!white] (6,5) rectangle (9,8);
        \draw[fill=black!100!white] (9,4) rectangle (10,5);
        \draw[fill=black!75!white] (7,6) rectangle (8,7);
        \draw[line width=0.2mm, draw=black, fill=black!20!white] (4,0) grid  (15,9);
        \node[white] at (9.5, 4.5) {$\pmb{K}$};
        \node[white] at (7.5, 6.5) {$\pmb{G_1}$};
    \end{tikzpicture}
    \caption{\small Illustration of the element~$K$ with patch~$\Nb[K][3]$ and element~$G_1$ with patch~$\Nb[G_1][1]$. Here we have~$\distance(G_1,K)=\mu = 2$.}\label{fig:patches}
\end{figure}

\begin{remark}\label{rem:enriched_localization}
    In the proof of \Cref{thm:localization}, we split each of the operators~$\addCorrector$, $\addCorrector^\mathrm{loc}$, and~$1-\Corrector^{[\ell]}$ into their element-wise contributions. This leads to chains of elements that lie within respective patches, such that the definitions for each~$\addCorrector$ and~$\addCorrector^\mathrm{loc}$ differ strongly. 
    To illustrate this, we consider the term~$\zeta = \addCorrector^\mathrm{loc}(1-\Corrector^{[\ell]})v$, and~\Cref{fig:patches}. We have 
    \begin{equation*}
        \zeta = \sum_{K\in\mathcal{T}_H} \addCorrector^\mathrm{loc}(1\vert_K - \Corrector_K^{[\ell]})v \eqqcolon \sum_{K\in\mathcal{T}_H} \addCorrector^\mathrm{loc}\zeta_K, 
    \end{equation*}
    where~$\supp (\zeta_K) \subset\Nb$. In the next step, we fix one element~$K\in\mathcal{T}_H$ and have 
    \begin{equation*}
        \addCorrector^\mathrm{loc}\zeta_K = \sum_{G_1\in\Nb} \addCorrector_{G_1}^{[\lambda_{G_1}(K)]} \zeta_K. 
    \end{equation*}
    This now is the crucial step. For each element~$K$, we define the localized enriched correction operator for each element~$G_1\in\Nb$ separately. Thus, every~$\addCorrector^\mathrm{loc}$ is defined differently for each~$K$. Due to the definition of the localized enriched correction operator~\eqref{eq:loc_enriched_corrector_def}, we may choose each~$\lambda_{G_1}(K)$ appropriately as follows. For~$G_1\in\Nb$ with~$\distance(G_1,K)=\mu$, we have that the correction operator~$\Corrector^{[\ell]}_K$ has already decayed with a remainder of order~$\mathcal{O}{(\exp^{-C\mu})}$. This allows us to define the localized element-wise enriched correction on a patch of size~$\ell-\mu$ around~$G_1$, as the total decay would add up to~$\mathcal{O}(\exp^{-C(\ell-\mu)}) \cdot \mathcal{O}(\exp^{-C\mu}) = \mathcal{O}(\exp^{-C\ell})$. We note that this idea works similarly for higher exponents of~$\addCorrector^\mathrm{loc}$ and~$\addCorrector$. 

    This is also the reason why we compute the basis functions simply on the patch as in~\eqref{eq:loc_enriched_basis}. Since all localized element-wise enriched correction operators are defined inside the same patch~$\Nb$, increasing the computational domain of them to exactly the patch does not introduce additional errors as the decay result still holds.
\end{remark}}%

In order to be able to apply the decay results of the element-wise enriched correction operators (cf.~\Cref{lem:DG_decay} and~\ref{lem:DG_decay_error}), we need the following auxiliary result {for corrections of arbitrary functions~$v\in\Hloc$}. 
\begin{lemma}\label{lem:DG_decay_local}
    Let~$G\in\mathcal{T}_H$ and~$\lambda_G\in\N$. Then, for any function~$v\in\Hloc$ and~$w\in W$ we have 
    \begin{equation*}
        \HIP{(\addCorrector_G^{[\lambda_G]} - \addCorrector_G)v}{w} \lesssim H^2\exp^{-C\lambda_G}\Norm{\nabla v}[L^2(G)] \Norm{\nabla w}[L^2(\mathtt{N}^{\lambda_G+1}(G))]. 
    \end{equation*}
\end{lemma}
\begin{proof}
    First, define a cut-off function~$\eta$ by 
    \begin{equation}\label{eq:cutoff_Glambda}
    \begin{aligned}
        \eta &\equiv 0,\qquad&&\mathrm{in~}\mathtt{N}^{\lambda_G}(G),\\
        \eta &\equiv 1,\qquad&&\mathrm{in~}\Omega\setminus\mathtt{N}^{\lambda_G+1}(G),\\
        0\leq\eta&\leq 1,\quad\Norm{\nabla\eta}[L^\infty(R)] \lesssim H^{-1},\qquad&&\mathrm{in~}R=\mathtt{N}^{\lambda_G+1}(G)\setminus\mathtt{N}^{\lambda_G}(G).
    \end{aligned}
    \end{equation}
    Recall the definition of the bubble operator~$\bubbleOperator$ from~\Cref{subsec:corrections}. Employing the cut-off function~\eqref{eq:cutoff_Glambda} we have for any~$w\in W$ the identity
    \begin{equation}\label{eq:DGloc_DG_eta}
        \HIP{(\addCorrector_G^{[\lambda_G]} - \addCorrector_G)v}{(1-\mathcal{B}_H)(\eta w)} = -\HIP{\addCorrector_Gv}{(1-\mathcal{B}_H)(\eta w)} = 0.
    \end{equation}
    Using~$\bubbleOperator w=0$ for~$w\in W$ and~\eqref{eq:DGloc_DG_eta} in the equality, the stability estimates~\eqref{eq:stability_H1_bubble_op} and~\eqref{eq:cutoff_Glambda} similar to~\cite[eq.~(A.2)]{DonHM23} in the first inequality, and~\Cref{lem:DG_decay_error} in the second inequality, we obtain 
    \begin{equation}\label{eq:DGlocminDG_estimate}
        \begin{aligned}
        \HIP{(\addCorrector_G^{[\lambda_G]} - \addCorrector_G)v}{w} &= \HIP{(\addCorrector_G^{[\lambda_G]} - \addCorrector_G)v}{(1-\mathcal{B}_H)((1-\eta)w)}\\
        &\lesssim \Norm{\nabla (\addCorrector_G^{[\lambda_G]} - \addCorrector_G)v} \Norm{\nabla w}[L^2(\mathtt{N}^{\lambda_G+1}(G))]\\
        &\lesssim \exp^{-C\lambda_G}\Norm{\nabla \addCorrector_Gv}\Norm{\nabla w}[L^2(\mathtt{N}^{\lambda_G+1}(G))]. 
        \end{aligned}
    \end{equation}
    Furthermore, we have 
    \begin{equation}\label{eq:DG_estimate}
        \Norm{\nabla \addCorrector_G v}^2 \lesssim \HIP{\addCorrector_G v}{\addCorrector_G v} = -\IP{v}{\addCorrector_G v}[L^2(G)] \lesssim H^2\Norm{\nabla v}[L^2(G)]\Norm{\nabla \addCorrector_G v}. 
    \end{equation}
    Combining~\eqref{eq:DG_estimate} with~\eqref{eq:DGlocminDG_estimate} leads to 
    \begin{equation*}
        \HIP{(\addCorrector_G^{[\lambda_G]} - \addCorrector_G)v}{w} \lesssim H^2 \exp^{-C\lambda_G}\Norm{\nabla v}[L^2(G)] \Norm{\nabla w}[L^2(\mathtt{N}^{\lambda_G+1}(G))]. 
    \end{equation*}
\end{proof}

For a clearer presentation in the following proof we simplify localization parameters and patches, when it is clear how they are defined. We abbreviate~$\lambda_i \coloneqq \lambda_{G_{i}}(G_{i-1})$ for any~$i=1,\dots,\nu$, where we formally set~$G_0\coloneqq K$, and further~$\lambda_0 \coloneqq \ell$. Next, we also abbreviate the patches~$\omega_i\coloneqq \Nb[G_{i-1}][\lambda_{i-1}]$, and~$\omega_i^{+1}\coloneqq \Nb[G_{i-1}][\lambda_{i-1}+1]$. 

\begin{proof}[Proof of~\Cref{thm:localization}]
    {In the first part we adapt the representation of the localization error into a form such that \Cref{lem:DG_decay}, \Cref{lem:DG_decay_error}, and \Cref{lem:DG_decay_local} can be applied. }%
    Recall the definition of the enriched correction operator~\eqref{eq:tildeC_def}, its localized counterpart~\eqref{eq:loc_enriched_corrector_def}, and the (localized) correction operator~\eqref{eq:CK_definition}. We have 
    \begin{multline}\label{eq:localization_error_split}
        \Norm{\nabla ((\addCorrector^\mathrm{loc})^\nu - \addCorrector^\nu)(1-\Corrector^{[\ell]})v}\\
        \leq \Norm{\nabla (\addCorrector^\mathrm{loc} - \addCorrector)(\addCorrector^\mathrm{loc})^{\nu-1}(1-\Corrector^{[\ell]})v} + \Norm{\nabla \addCorrector((\addCorrector^\mathrm{loc})^{\nu-1} - \addCorrector^{\nu-1})(1-\Corrector^{[\ell]})v} . 
    \end{multline}
    The first term can be estimated using the localization estimate~\Cref{lem:CK_decay_error}. Next, we estimate the second term. For any~$w\in W$ we have using~\eqref{eq:Pi_estimate} 
    \begin{multline*}\label{eq:localization_error_recursion_initial}
        \HIP{\addCorrector((\addCorrector^\mathrm{loc})^{\nu-1} - \addCorrector^{\nu-1})(1-\Corrector^{[\ell]})v}{w} = -\IP{((\addCorrector^\mathrm{loc})^{\nu-1} - \addCorrector^{\nu-1})(1-\Corrector^{[\ell]})v}{w}[\LL][\big]\\
        \begin{aligned}
          \lesssim H^2\Norm{\nabla ((\addCorrector^\mathrm{loc})^{\nu-1} - \addCorrector^{\nu-1})(1-\Corrector^{[\ell]})v}\,\Norm{\nabla w}.
        \end{aligned}
    \end{multline*}
    Choosing~$w=\addCorrector((\addCorrector^\mathrm{loc})^{\nu-1} - \addCorrector^{\nu-1})(1-\Corrector^{[\ell]})v$, using the estimate~$\Norm{\nabla w}^2\lesssim \HIP{w}{w}$, and dividing by~$\Norm{\nabla w}$ we get 
    \begin{equation}\label{eq:localization_error_recursion_initial_estimate}
        \Norm{\nabla \addCorrector((\addCorrector^\mathrm{loc})^{\nu-1} - \addCorrector^{\nu-1})(1-\Corrector^{[\ell]})v} \lesssim H^2\Norm{\nabla ((\addCorrector^\mathrm{loc})^{\nu-1} - \addCorrector^{\nu-1})(1-\Corrector^{[\ell]})v}. 
    \end{equation}
    The norm on the right-hand side can now be split analogously to~\eqref{eq:localization_error_split} and with a similar estimate as in~\eqref{eq:localization_error_recursion_initial_estimate} we have a recursion through to the last step 
    \begin{equation}\label{eq:localization_error_recursion_final_estimate}
        \begin{aligned}
            \Norm{\nabla \addCorrector(\addCorrector^\mathrm{loc} - \addCorrector)(1-\Corrector^{[\ell]})v} &\lesssim H^2\Norm{\nabla (\addCorrector^\mathrm{loc} - \addCorrector)(1-\Corrector^{[\ell]})v}. 
        \end{aligned}
    \end{equation}
    Employing the recursive argument and the final estimate~\eqref{eq:localization_error_recursion_final_estimate} in~\cref{eq:localization_error_split} we get  
    \begin{equation}\label{eq:localization_error_intermediate}
        \Norm{\nabla ((\addCorrector^\mathrm{loc})^\nu - \addCorrector^\nu)(1-\Corrector^{[\ell]})v}
        \lesssim \sum_{i = 1}^{\nu} H^{2i-2} \Norm{\nabla (\addCorrector^\mathrm{loc} - \addCorrector)(\addCorrector^\mathrm{loc})^{\nu-i}(1-\Corrector^{[\ell]})v}. 
    \end{equation}
    Thus, we have to bound terms of the form~$\Norm{\nabla (\addCorrector^\mathrm{loc} - \addCorrector)(\addCorrector^\mathrm{loc})^{i}(1-\Corrector^{[\ell]})v}$ (after performing an index shift) for any~$i=0,\dots,\nu-1$. %
    {For these terms, we are now able to use the previous auxiliary results \Cref{lem:DG_decay}, \Cref{lem:DG_decay_error}, and \Cref{lem:DG_decay_local}. }%
    We use the notation from the start of the section and apply the definition of the (localized) element-wise (enriched) correction operators to obtain with~$w = (\addCorrector^\mathrm{loc} - \addCorrector)(\addCorrector^\mathrm{loc})^{i}(1-\Corrector^{[\ell]})v$ the following sum 
    \begin{multline*}
        \Norm{\nabla (\addCorrector^\mathrm{loc} - \addCorrector)(\addCorrector^\mathrm{loc})^{i}(1-\Corrector^{[\ell]})v}^2\lesssim \HIP{(\addCorrector^\mathrm{loc} - \addCorrector)(\addCorrector^\mathrm{loc})^{i}(1-\Corrector^{[\ell]})v}{w}\\
        \begin{aligned}
          &= \sum_{K\in\mathcal{T}_H}\sum_{G_1\in\omega_1}\cdots\sum_{G_{i+1}\in\omega_{i+1}} \HIP{(\addCorrector_{G_{i+1}}^{[\lambda_{i+1}]}-\addCorrector_{G_{i+1}})\addCorrector_{G_{i}}^{[\lambda_{i}]}\cdots\addCorrector_{G_{1}}^{[\lambda_1]}(v\vert_K-\Corrector_K^{[\ell]} v)}{w}[\big]. 
        \end{aligned}
    \end{multline*}
    {Here, we make use of the definition of the localization parameter~$\lambda_i$ for all~$i$. In every step, the localization error of the $i$th (localized) enriched correction operator scales like~$\mathcal{O}(\operatorname{exp}(-C\lambda_i))$ and is then multiplied by the localization error of the $(i-1)$th enriched correction operator. The combined localization error is then recursively multiplied such that by choice of the~$\lambda_i$ the overall localization error is of order~$\mathcal{O}(\operatorname{exp}(-C\ell))$. There, we use the fact that we have only a finite overlap of patches and that the norms of the enriched corrections scale with~$H^2$. }%
    In this sum we make use of the definitions of each~$\lambda_i$ in~\Cref{thm:localization} and~\Cref{lem:DG_decay_error} to bound the localization error for every term in the sum by an exponential term. 
    With~$w = (\addCorrector^\mathrm{loc} - \addCorrector)(\addCorrector^\mathrm{loc})^{i}(1-\Corrector^{[\ell]})v$, by~\Cref{lem:DG_decay_local} in the first estimate and applying the discrete Cauchy-Schwarz inequality to the innermost sum in the second estimate, we obtain 
    \begin{multline}\label{eq:DlocminD_Dlocnu_sum_estimate_initial}
        \Norm{\nabla (\addCorrector^\mathrm{loc} - \addCorrector)(\addCorrector^\mathrm{loc})^{i}(1-\Corrector^{[\ell]})v}^2\\
        \begin{alignedat}{3}
            &\lesssim \!\!\sum_{K\in\mathcal{T}_H}\sum_{G_1\in\omega_1} \!\!\cdots\hspace{-.45cm}\sum_{G_{i+1}\in\omega_{i+1}}\hspace{-.4cm} \mathrlap{H^2 \exp^{-C\lambda_{i+1}}\Norm{\nabla \addCorrector_{G_{i}}^{[\lambda_{i}]}\cdots\addCorrector_{G_{1}}^{[\lambda_1]}(v\vert_K-\Corrector_K^{[\ell]} v)}[L^2(G_{i+1})] \Norm{\nabla w}[L^2(\omega_{i+2}^{+1})]} && \\
            &\lesssim\! H^2 \!\!\sum_{K\in\mathcal{T}_H}\sum_{G_1\in\omega_1}\!\! \cdots \!\!\sum_{G_{i}\in\omega_{i}} && \Big(\sum_{G_{i+1}\in\omega_{i+1}} \hspace{-.35cm}\exp^{-2C\lambda_{i+1}}\Norm{\nabla \addCorrector_{G_{i}}^{[\lambda_{i}]}\cdots\addCorrector_{G_{1}}^{[\lambda_1]}(v\vert_K-\Corrector_K^{[\ell]} v)}[L^2(G_{i+1})]^2 \Big)^{\frac12} \\
            & && \cdot \Big(\sum_{G_{i+1}\in\omega_{i+1}}\Norm{\nabla w}[L^2(\omega_{i+2}^{+1})]^2 \Big)^{\frac12}. &&
        \end{alignedat}
    \end{multline}
    We fix~$G_{i}\in\omega_{i}$, such that~$\lambda_{i+1}$ and~$\omega_{i+1}$ are properly defined. Then define the rings~$R^\mu$ around the element~$G_{i}$ by $R^\mu = \{G\in\mathcal{T}_H \mid \distance(G,G_{i}) = \mu\}$. In the following estimate we employ the definitions of~$\omega_{i+1}$, $R^\mu$, and~$\lambda_{i+1}$ and the Cauchy-Schwarz and Young inequalities to the first factor inside the sum of~\eqref{eq:DlocminD_Dlocnu_sum_estimate_initial} to obtain 
    \begin{subequations}\label{eq:DlocminD_Dlocnu_sum_discreteCS}
    \begin{multline}
        \sum_{G_{i+1}\in\omega_{i+1}} \exp^{-2C\lambda_{i+1}}\Norm{\nabla \addCorrector_{G_{i}}^{[\lambda_{i}]}\cdots\addCorrector_{G_{1}}^{[\lambda_1]}(v\vert_K-\Corrector_K^{[\ell]} v)}[L^2(G_{i+1})]^2\\
        \begin{alignedat}{2}
            &\lesssim \sum_{\mu=0}^{\lambda_{i}} \exp^{-2C(\lambda_{i} - \mu)} \Big[&&\Norm{\nabla \addCorrector_{G_{i}} \addCorrector_{G_{i-1}}^{[\lambda_{i-1}]}\cdots\addCorrector_{G_{1}}^{[\lambda_1]}(v\vert_K-\Corrector_K^{[\ell]} v)}[L^2(R^\mu)]^2\\
            & && + \Norm{\nabla (\addCorrector_{G_{i}}^{[\lambda_{i}]} - \addCorrector_{G_{i}}) \addCorrector_{G_{i-1}}^{[\lambda_{i-1}]} \cdots \addCorrector_{G_{1}}^{[\lambda_1]}(v\vert_K-\Corrector_K^{[\ell]} v)}[L^2(R^\mu)]^2\Big]\eqqcolon L_1.
        \end{alignedat}
    \end{multline}
    Next, we apply~\Cref{lem:DG_decay} and~\ref{lem:DG_decay_error} and obtain 
    {\begin{equation}
        \begin{alignedat}{2}
            L_1&\lesssim \sum_{\mu=0}^{\lambda_{i}} \exp^{-2C(\lambda_{i} - \mu)} \Big[&&\exp^{-2C\mu} \Norm{\nabla \addCorrector_{G_{i}} \addCorrector_{G_{i-1}}^{[\lambda_{i-1}]}\cdots\addCorrector_{G_{1}}^{[\lambda_1]}(v\vert_K-\Corrector_K^{[\ell]} v)}^2\\
            & && + \exp^{-2C\lambda_{i}} \Norm{\nabla \addCorrector_{G_{i}} \addCorrector_{G_{i-1}}^{[\lambda_{i-1}]} \cdots \addCorrector_{G_{1}}^{[\lambda_1]}(v\vert_K-\Corrector_K^{[\ell]} v)}^2\Big]\eqqcolon L_2.
        \end{alignedat}
    \end{equation}}
    Finally, we can cancel the exponential terms and sum them up in the first estimate, and then use the definition of the element-wise enriched correction (see~\eqref{eq:DG_estimate}) in the second estimate to obtain 
    \begin{equation}
        \begin{alignedat}{1}
            L_2&\lesssim \lambda_{i} \exp^{-2C\lambda_{i}} \Norm{\nabla \addCorrector_{G_{i}} \addCorrector_{G_{i-1}}^{[\lambda_{i-1}]} \cdots \addCorrector_{G_{1}}^{[\lambda_1]}(v\vert_K-\Corrector_K^{[\ell]} v)}^2 \\
            &\lesssim H^4\lambda_{i} \exp^{-2C\lambda_{i}} \Norm{\nabla \addCorrector_{G_{i-1}}^{[\lambda_{i-1}]} \cdots \addCorrector_{G_{1}}^{[\lambda_1]}(v\vert_K-\Corrector_K^{[\ell]} v)}[L^2(G_{i})]^2. 
        \end{alignedat}
    \end{equation}
    \end{subequations}
    Using~\eqref{eq:DlocminD_Dlocnu_sum_discreteCS} in~\eqref{eq:DlocminD_Dlocnu_sum_estimate_initial}, the finite overlap of patches~$\Nb[G_{i+1}][\lambda_{i+1} + 1]\subset \Nb[G_{i}][\lambda_{i} + 1]$ and the fact, that~$\lambda_i\leq \ell$ for~$i=0,\dots,\nu$ we obtain with~$w = (\addCorrector^\mathrm{loc} - \addCorrector)(\addCorrector^\mathrm{loc})^{i}(1-\Corrector^{[\ell]})v$ 
    \begin{multline}\label{eq:DlocminD_Dlocnu_sum_estimate_first}
        \Norm{\nabla (\addCorrector^\mathrm{loc} - \addCorrector)(\addCorrector^\mathrm{loc})^{i}(1-\Corrector^{[\ell]})v}^2\\
        \begin{alignedat}{2}
            &\lesssim H^2\ell^{\frac{d+1}{2}} \hspace{-.2cm} \sum_{K\in\mathcal{T}_H}\sum_{G_1\in\omega_1}\hspace{-.1cm}\cdots\hspace{-.1cm}\sum_{G_{i}\in\omega_{i}} && H^2\exp^{-C\lambda_{i}} \Norm{\nabla \addCorrector_{G_{i-1}}^{[\lambda_{i-1}]}\cdots\addCorrector_{G_{1}}^{[\lambda_1]}(v\vert_K-\Corrector_K^{[\ell]} v)}[L^2(G_{i})] \\
            & &&\cdot \Norm{\nabla w}[L^2(\omega_{i+1}^{+1})].
        \end{alignedat}
    \end{multline}
    We have that the right-hand side of~\eqref{eq:DlocminD_Dlocnu_sum_estimate_first} has the same structure (up to a factor) as~\eqref{eq:DlocminD_Dlocnu_sum_estimate_initial} with one sum less. Thus, we can recursively apply analogous ideas to~\eqref{eq:DlocminD_Dlocnu_sum_discreteCS} and~\eqref{eq:DlocminD_Dlocnu_sum_estimate_first} in the first estimate and the discrete Cauchy-Schwarz inequality to finally get 
    \begin{multline}\label{eq:DlocminD_Dlocnu_sum_estimate_final}
        \Norm{\nabla (\addCorrector^\mathrm{loc} - \addCorrector)(\addCorrector^\mathrm{loc})^{i}(1-\Corrector^{[\ell]})v}^2\\
        \begin{alignedat}{1}
            &\mathrlap{\lesssim H^{2i}\ell^{\frac{d+1}{2}i} \sum_{K\in\mathcal{T}_H} \sum_{G_1\in\omega_1} H^2\exp^{-C\lambda_1} \Norm{\nabla (v\vert_K-\Corrector_K^{[\ell]} v)}[L^2(G_{1})] \Norm{\nabla w}[L^2(\omega_2^{+1})]} \\
            &\lesssim\! H^{2({i+1})}\ell^{\frac{d+1}{2}i} \!\sum_{K\in\mathcal{T}_H} \!\Big(\!\sum_{G_1\in\omega_1}\! \exp^{-2C\lambda_1} \Norm{\nabla (v\vert_K-\Corrector_K^{[\ell]} v)}[L^2(G_{1})]^2\Big)^{\frac12} \Big(\sum_{G_1\in\omega_1} \Norm{\nabla w}[L^2(\omega_2^{+1})]^2\Big)^{\frac12}.
        \end{alignedat}
    \end{multline}
    As above, we fix~$K\in\mathcal{T}_H$ and define the rings~$R^\mu$ with~$\mu$ layers around~$K$. Then we use the definition of the ring~$R^\mu$ and~$\lambda_1$ in the first estimate, the Cauchy-Schwarz and Young inequalities twice in the second estimate, \Cref{lem:CK_decay} and~\ref{lem:CK_decay_error} in the third estimate, summing up all terms in the fourth estimate, and using the definition of the element-wise correction~\eqref{eq:CK_definition} in the least estimate. Altogether, we obtain 
    \begin{multline}\label{eq:OneminCell_sum_discreteCS}
        \sum_{G_1\in\Nb} \exp^{-2C\lambda_1}\Norm{\nabla (v\vert_K-\Corrector_K^{[\ell]} v)}[L^2(G_{1})]^2\lesssim \sum_{\mu=0}^{\ell} \mathrlap{\exp^{-2C(\ell -\mu)} \Norm{\nabla (v\vert_K-\Corrector_K^{[\ell]} v)}[L^2(R^\mu)]^2} \\
        \begin{alignedat}{1}
            &\lesssim {\exp^{-2C\ell} \Norm{\nabla v}[L^2(K)]^2} + \sum_{\mu=0}^\ell \exp^{-2C(\ell -\mu)} \Big[\Norm{\nabla (\Corrector_K v-\Corrector_K^{[\ell]} v)}[L^2(R^\mu)]^2 + \Norm{\nabla \Corrector_K v}[L^2(R^\mu)]^2\Big]\\
            &\lesssim {\exp^{-2C\ell} \Norm{\nabla v}[L^2(K)]^2} + \sum_{\mu=0}^\ell \exp^{-2C(\ell -\mu)} \Big[\exp^{-2C\ell} \Norm{\nabla (\Corrector_K v)}^2 + \exp^{-2C\mu}\Norm{\nabla \Corrector_K v}^2\Big]\\
            &\lesssim {\exp^{-2C\ell} \big[\Norm{\nabla v}[L^2(K)]^2 + \ell\Norm{\nabla \Corrector_Kv}^{{2}}\big]} \lesssim {\ell\exp^{-2C\ell} \Norm{\nabla v}[L^2(K)]^2.} 
        \end{alignedat}
    \end{multline}
    Employing the discrete Cauchy-Schwarz inequality twice in~\eqref{eq:DlocminD_Dlocnu_sum_estimate_final} and using~\eqref{eq:OneminCell_sum_discreteCS} we obtain with the finite overlap of patches (similar to above)
    \begin{multline*}
        \Norm{\nabla (\addCorrector^\mathrm{loc} - \addCorrector)(\addCorrector^\mathrm{loc})^{i}(1-\Corrector^{[\ell]})v}^2\\
        \begin{aligned}
            &\lesssim H^{2(i+1)}\ell^{\frac{d+1}{2}(i+1)} \!\sum_{K\in\mathcal{T}_H}\! \exp^{-C\ell} \Norm{\nabla v}[L^2(K)] \Norm{\nabla (\addCorrector^\mathrm{loc} - \addCorrector)(\addCorrector^\mathrm{loc})^{i}(1-\Corrector^{[\ell]})v}[L^2(\mathtt{N}^{\ell+1}(K))]\\
            &\lesssim H^{2(i+1)}\ell^{\frac{d+1}{2}(i+1)}\ell^{\frac{d}{2}} \exp^{-C\ell} \Norm{\nabla v} \Norm{\nabla (\addCorrector^\mathrm{loc} - \addCorrector)(\addCorrector^\mathrm{loc})^{i}(1-\Corrector^{[\ell]})v}. 
        \end{aligned}
    \end{multline*}
    Dividing by $\Norm{\nabla (\addCorrector^\mathrm{loc} - \addCorrector)(\addCorrector^\mathrm{loc})^{i}(1-\Corrector^{[\ell]})v}$, using that~$H^{2(i+1)}\ell^{\frac{d+1}{2}(i+1)}\ell^{\frac{d}{2}}\lesssim 1$ by assumption, and plugging it into~\eqref{eq:localization_error_intermediate} yields the assertion 
    \begin{equation*}
        \Norm{\nabla ((\addCorrector^\mathrm{loc})^\nu - \addCorrector^\nu)(1-\Corrector^{[\ell]})v}
        \lesssim \sum_{i = 1}^{\nu} H^{2i-2} \exp^{-C\ell} \Norm{\nabla v} \lesssim \exp^{-C\ell} \Norm{\nabla v}.
    \end{equation*}
\end{proof}

\subsection{Proof of the main theorem}\label{subsec:proof_main}
To simplify the presentation in the following proofs assume that~$k>1$, such that~$j>0$. For the case~$k=1$, see~\Cref{rem:classical_holod} at the end of this section. 

Let~$u$ denote the exact solution to~\eqref{eq:heat_model_problem}, and~$\check{u}_H$ the eho-LOD solution to~\eqref{eq:practical_method}. We define the error between the exact solution and its map into the multiscale space by~$\varphi(t)=u(\cdot, t)-(\tProjection u(\cdot, t) + \hProjection u(\cdot, t))$, and the localized version is denoted with~$\varphi^\mathrm{loc}(t)=u(\cdot, t)-(\tProjectionloc u(\cdot, t) + \hProjectionloc u(\cdot, t))$. We prove the theorem in three steps. First we estimate the full error~$e(t) = u(\cdot, t)-\check{u}_H(t)$ by the map into the enriched multiscale space~$\varphi$ and a term that can be estimated using the localization results from~\Cref{subsec:corrections}. The second step is the estimation of the localization error, and then the third step consists of bounding~$\varphi$. 

\begin{proof}[Proof of~\Cref{thm:semi_discrete}]
  \textbf{1.}~We start by bounding the error by the map into the enriched multiscale space~$\varphi$ and a localization error. First, we have by Galerkin orthogonality that the error $e$ solves the following equations, where the estimates follow from Cauchy-Schwarz and Young's inequality, where we omit the argument~$t\in [0,T]$ to improve readability, 
  \begin{equation}\label{eq:error_var_form}
      \begin{aligned}
          \IP{\diff t e}{e}+\HIP{e}{e} &= \IP{\diff t e}{\varphi^\mathrm{loc}} + \HIP{e}{\varphi^\mathrm{loc}}\\
          &\hspace{-1cm}\leq \tfrac{1}{2} (\Norm{\diff t e}^2 + \Norm{\varphi^\mathrm{loc}}^2 + \HIP{e}{e} + \HIP{{\varphi^\mathrm{loc}}}{{\varphi^\mathrm{loc}}}),\\
          \IP{\diff t e}{\diff t e}+\HIP{e}{\diff t e} &= \IP{\diff t e}{\diff t \varphi^\mathrm{loc}} + \HIP{e}{\diff t \varphi^\mathrm{loc}}\\
          &\hspace{-1cm}\leq \tfrac{1}{2} (\Norm{\diff t e}^2 + \Norm{\diff t \varphi^\mathrm{loc}}^2 + \HIP{e}{e} + \HIP{{\diff t \varphi^\mathrm{loc}}}{{\diff t \varphi^\mathrm{loc}}}), 
      \end{aligned}
  \end{equation}
  If we sum up the inequalities~\eqref{eq:error_var_form} and integrate over time from~$0$ to~$t$, we have 
  \begin{multline*}
      \int\limits_0^t \tfrac{1}{2}\diff t \Norm{e(s)}^2 + \HIP{e(s)}{e(s)} + \Norm{\diff t e(s)} + \tfrac{1}{2}\diff t \HIP{e(s)}{e(s)} \d s\\[-0.3\baselineskip]
      \begin{aligned}
          &\leq \int\limits_0^t \Norm{\diff t e(s)}^2 + \tfrac{1}{2}(\Norm{\varphi^\mathrm{loc}(s)}^2 + \Norm{\diff t \varphi^\mathrm{loc}(s)}^2) \d s\\[-0.3\baselineskip]
          &\quad +\int\limits_0^t \HIP{e(s)}{e(s)} + \tfrac{1}{2}(\HIP{\varphi^\mathrm{loc}(s)}{\varphi^\mathrm{loc}(s)} + \HIP{\diff t \varphi^\mathrm{loc}(s)}{\diff t \varphi^\mathrm{loc}(s)}) \d s. 
      \end{aligned}
  \end{multline*}
  Re-arranging terms on both sides yields  
  \begin{multline}\label{eq:error_estimate_philoc}
    \Norm{e(t)}^2 - \Norm{e(0)}^2 + \HIP{e(t)}{e(t)} - \HIP{e(0)}{e(0)}\\
    \leq \!\!\int\limits_0^t\! \Norm{\varphi^\mathrm{loc}(s)}^2 + \Norm{\diff t \varphi^\mathrm{loc}(s)}^2 + \HIP{\varphi^\mathrm{loc}(s)}{\varphi^\mathrm{loc}(s)} + \HIP{\diff t \varphi^\mathrm{loc}(s)}{\diff t \varphi^\mathrm{loc}(s)} \d s. 
  \end{multline}
  By the definition of the initial condition~\eqref{eq:sd_init_condition_practical} we have~$e(0)=\varphi^\mathrm{loc}(0)$. In the following we omit the argument~$s\in [0,T]$ to improve readability, and note that the following equations hold for all~$s$. Since~$\varphi^\mathrm{loc}\in W$ we have 
  \begin{equation}\label{eq:philoc_gradient_estimate}
    \Norm{\varphi^\mathrm{loc}} \lesssim H\Norm{\nabla \varphi^\mathrm{loc}}, 
  \end{equation}
  which holds similarly for~$\diff t \varphi^\mathrm{loc}$. Thus, we bound terms of the form 
  \begin{equation}\label{eq:philoc_estimate}
      \Norm{\nabla \varphi^\mathrm{loc}} \leq \Norm{\nabla \varphi} + \Norm{\nabla ((\tProjection + \hProjection)u-(\tProjectionloc + \hProjectionloc)u)}, 
  \end{equation}
  and estimates for~$\nabla\diff t \varphi^\mathrm{loc}$ follow analogously (we just need to keep in mind that more temporal regularity is required). The first term on the right-hand side of~\eqref{eq:philoc_estimate} will be estimated in the last step. 
  
  \textbf{2.}~For the second term on the right-hand side of~\eqref{eq:philoc_estimate} we use the Cauchy-Schwarz inequality to split the error into the localization error of the correction operator, the localization error of the enriched correction operator and a third error which is a defect resulting from the definition of the map into the global enriched correction space~\eqref{eq:add_proj_def}. This leads to 
  \begin{multline}\label{eq:cPloc_error}
      \Norm{\nabla ((\tProjection + \hProjection)u-(\tProjectionloc + \hProjectionloc)u)} \lesssim \Norm{\nabla (\Corrector - \Corrector^{[\ell]})u}\\
      + \Norm{\nabla \sum_{\nu=1}^{j}((\addCorrector^{\mathrm{loc}})^\nu - \addCorrector^\nu)(\diff t[\nu] \tProjectionloc u)}[\LL][\Big] + \Norm{\nabla \sum_{\nu=1}^{j}\addCorrector^\nu(\diff t[\nu] (\Corrector^{[\ell]} - \Corrector) u)}[\LL][\Big]. 
  \end{multline}
  The first term can directly be estimated employing~\Cref{lem:C_decay_error}, and the second term is bounded by~\Cref{thm:localization} using that~$\ell\sim C_p|\log H|$ and thus~$H^2\ell^{d+1}\lesssim 1$ holds. For the estimation of the third term we also want to use~\Cref{lem:C_decay_error} but need some more preliminary considerations. Employing the definition of~$\addCorrector$ we obtain for any~$w\in W$ 
  \begin{multline}\label{eq:Dnu_CminCell_first_step}
      \HIP{\sum_{\nu=1}^{j}(\addCorrector^\nu)(\diff t[\nu] (\Corrector^{[\ell]} - \Corrector) u)}{w}[\Big]\\[-2ex]
      \begin{aligned}
          &= -\IP{\diff t (\Corrector^{[\ell]} - \Corrector)u}{w} - \IP{\diff t \sum_{\nu=2}^{j}\addCorrector^{\nu-1}(\diff t[\nu-1] (\Corrector^{[\ell]} - \Corrector) u)}{w}[\LL][\Big]\\
          &\lesssim H^2\Norm{\nabla \diff t (\Corrector^{[\ell]} - \Corrector)u}\,\Norm{\nabla w}\\
          &\quad + H^2\Norm{\nabla \diff t \sum_{\nu=2}^{j}\addCorrector^{\nu-1}(\diff t[\nu-1] (\Corrector^{[\ell]} - \Corrector) u)}[\LL][\Big]\,\Norm{\nabla w}.
      \end{aligned}
  \end{multline}
  We choose~$w=\sum_{\nu=1}^{j}(\addCorrector^\nu)(\diff t[\nu] (\Corrector^{[\ell]} - \Corrector) u)$. Using that~$\Norm{\nabla w}^2\lesssim \HIP{w}{w}$, dividing by~$\Norm{\nabla w}$, and employing~\Cref{lem:C_decay_error} we obtain 
  \begin{multline}\label{eq:Dnu_CminCell_first_step_estimate}
  \Norm{\nabla \sum_{\nu=1}^{j}(\addCorrector^\nu)(\diff t[\nu] (\Corrector^{[\ell]} - \Corrector) u)}[\LL][\Big]\\
    \begin{aligned}
          &\lesssim H^2\exp^{-C\ell}\Norm{\nabla \diff t u} + H^2\Norm{\nabla \diff t \sum_{\nu=2}^{j}\addCorrector^{\nu-1}(\diff t[\nu-1] (\Corrector^{[\ell]} - \Corrector) u)}[\LL][\Big].
    \end{aligned}
  \end{multline}
  The first term on the right-hand side scales as desired rate with respect to~$\ell$, and the second term can be recursively estimated with similar arguments until the final estimate 
  \begin{equation}\label{eq:Dnu_CminCell_last_step_estimate}
      \Norm{\nabla \diff t[j-1] \sum_{\nu=j}^{j}(\addCorrector^{\nu-(j-1)})(\diff t[\nu-(j-1)] (\Corrector^{[\ell]} - \Corrector) u)}[\LL][\Big]\lesssim H^2\exp^{-C\ell}\Norm{\nabla {\diff t[j]} u}.
  \end{equation}
  Overall, employing~\eqref{eq:Dnu_CminCell_first_step_estimate} recursively in~\eqref{eq:Dnu_CminCell_first_step} with~\eqref{eq:Dnu_CminCell_last_step_estimate} we obtain 
  \begin{equation}\label{eq:Dnu_CminCell_complete_estimate}
      \Norm{\nabla \sum_{\nu=1}^{j}(\addCorrector^\nu)(\diff t[\nu] (\Corrector^{[\ell]} - \Corrector) u)}[\LL][\Big] \lesssim H^2\exp^{-C\ell} \sum_{\nu=1}^j H^{2(\nu-1)}\Norm{\nabla \diff t[\nu] u}. 
  \end{equation}
  The localization error~\eqref{eq:cPloc_error} can thus be bounded using~\Cref{lem:C_decay_error}, \Cref{thm:localization}, and~\eqref{eq:Dnu_CminCell_complete_estimate} by 
  \begin{equation}\label{eq:cPloc_estimate_complete}
      \Norm{\nabla ((\tProjection + \hProjection)u-(\tProjectionloc + \hProjectionloc)u)} \lesssim \exp^{-C\ell} C_\mathrm{data}. 
  \end{equation}

  \textbf{3.}~The next step is to estimate the terms~$\Norm{\nabla\varphi}$ from~\eqref{eq:philoc_estimate}. In this step we again omit the argument~$s\in[0,T]$ and note that the estimates holds for all~$s$. We have $\varphi\in W$, which yields using similar arguments as in Lemma~\ref{lem:error_addCorr} 
  \begin{equation}\label{eq:phi_estimate}
      \begin{aligned}
          \HIP{\varphi}{\varphi} &= \HIP{u}{\varphi} - \HIP{\hProjection u}{\varphi} = \IP{f}{\varphi} - \IP{\diff t u}{\varphi} - \HIP{\hProjection u}{\varphi}\\
          &\lesssim H^{p+2} \Norm{f}[{\Space{H}{\mathcal{T}_H}[k]}] \, \Norm{\nabla \varphi} - \IP{\diff t u}{\varphi} - \HIP{\hProjection u}{\varphi}.
      \end{aligned}
  \end{equation}
  Applying the estimate~$\Norm{\nabla \varphi}^2\lesssim \HIP{\varphi}{\varphi}$, the first term on the right-hands side yields the optimal rate. For the last two terms on the right-hand side of~\eqref{eq:phi_estimate} we have using Cauchy-Schwarz inequality and~$\varphi\in W$
  \begin{multline}\label{eq:dtu_hPu_estimate_initial}
      \IP{\diff t u}{\varphi} \! + \HIP{\hProjection u}{\varphi} = \IP{\diff t u}{\varphi} - \! \IP{\diff t \big(\tProjection u + \sum_{\nu=2}^{j}\addCorrector^{\nu-1}(\diff t[\nu-1] \tProjection u)\big)}{\varphi}[\LL][\Big]\\
        \begin{aligned}
          &\lesssim H^2\Norm{\nabla \diff t \Big(u-\tProjection u - \sum_{\nu=2}^{j}\addCorrector^{\nu-1}(\diff t[\nu-1] \tProjection u)\Big)}[\LL][\Big]\,\Norm{\nabla \varphi},
        \end{aligned}
  \end{multline}
  where the first term on the right-hand side (up to the factor~$H^2$) can be split into 
  \begin{multline}\label{eq:dtu_hPu_split}
      \Norm{\nabla \diff t (u-\tProjection u - \sum_{\nu=2}^{j}\addCorrector^{\nu-1}(\diff t[\nu-1] \tProjection u))}[\LL][\Big] \\
      \lesssim \Norm{\nabla \diff t (u-\tProjection u - \sum_{\nu=2}^{j+1}\addCorrector^{\nu-1}(\diff t[\nu-1] \tProjection u))}[\LL][\Big] + \Norm{\nabla \diff t \addCorrector^{j}(\diff t[j] \tProjection u)}. 
  \end{multline}
  We note here that we artificially expand the enriched multiscale space for the first term on the right-hand side to obtain a better convergence with respect to~$j$. Here we leverage that the enriched corrections, i.e., the second term on the right-hand side, have much better scaling than the enriched corrections from the left-hand side. This expansion of the space is not required in practice, and is only a theoretical device. This estimate, however, only works if there is sufficient temporal regularity of~$u$. 
  For the second term we have for any $w\in W$ 
  \begin{equation}\label{eq:DPu_first_step}
  \begin{aligned}
      \HIP{\diff t \addCorrector^{j}(\diff t[j] \tProjection u)}{w} &= \IP{\diff t[2] \addCorrector^{j-1}(\diff t[j-1] \tProjection u)}{w}\\
      &\lesssim H^2 \Norm{\nabla \diff t[2] \addCorrector^{j-1}(\diff t[j-1] \tProjection u)}\,\Norm{\nabla w}. 
  \end{aligned}
  \end{equation}
  With the choice~$w=\diff t \addCorrector^{j}(\diff t[j] \tProjection u)$ in~\eqref{eq:DPu_first_step} and~$\Norm{\nabla w}^2\lesssim \HIP{w}{w}$ we get  
  \begin{equation}\label{eq:DPu_first_step_estimate}
      \Norm{\nabla \diff t \addCorrector^{j}(\diff t[j] \tProjection u)} \lesssim H^2 \Norm{\nabla \diff t[2] \addCorrector^{j-1}(\diff t[j-1] \tProjection u)}. 
  \end{equation}
  The estimate in~\eqref{eq:DPu_first_step_estimate} can recursively be applied, which leads to 
  \begin{equation}\label{eq:DPu_final_estimate}
      \Norm{\nabla \diff t \addCorrector^{j}(\diff t[j]\tProjection u)}\lesssim H^{2j}\Norm{\nabla \diff t[j+1]\tProjection u}\lesssim H^{2j}\Norm{\nabla \diff t[j+1] u}. 
  \end{equation}
  As an intermediate step, starting from~\cref{eq:phi_estimate} using~$\Norm{\nabla \varphi}^2 \lesssim\HIP{\varphi}{\varphi}$, and equations~\eqref{eq:dtu_hPu_estimate_initial}, \eqref{eq:dtu_hPu_split}, and~\eqref{eq:DPu_final_estimate} we obtain 
  \begin{equation}\label{eq:phi_estimate_intermediate}
      \begin{aligned}
          \Norm{\nabla \varphi} &\lesssim H^{p+2}\Norm{f}[H^k(\mathcal{T}_H)] +H^{2j+2}\Norm{\nabla \diff t[j+1] u}\\
          &\quad + H^2 \Norm{\nabla \diff t (u-\tProjection u - \sum_{\nu=2}^{j+1}\addCorrector^{\nu-1}(\diff t[\nu-1] \tProjection u))}[\LL][\Big]. 
      \end{aligned}
  \end{equation}
  Finally, we bound the last term on the right-hand side of~\eqref{eq:phi_estimate_intermediate}. %
  {The idea here is based on the expansion given in~\eqref{eq:solution_approximation_expansion}. We can use the fact that $u$ is the solution to the parabolic equation and that $\addCorrector$ is defined to cancel time derivatives of $u$. Applying these properties iteratively and using projection properties of $\tProjection$ (in particular orthogonality), we in each step obtain a higher-order term based on regularity properties of $f$ and a remaining term that in each step is amplified by a factor $H^2$ until reaching the desired optimal convergence rate. More precisely,}  
  for any~$w\in W$ we have with ~\eqref{eq:Pi_estimate} using Cauchy-Schwarz inequality 
  \begin{multline*}
    \HIP{\diff t \big(u-\tProjection u - \sum_{\nu=2}^{j+1}\addCorrector^{\nu-1}(\diff t[\nu-1] \tProjection u)\big)}{w}[\Big] = \HIP{\diff t \big(u - \sum_{\nu=2}^{j+1}\addCorrector^{\nu-1}(\diff t[\nu-1] \tProjection u)\big)}{w}[\Big] \\
    \begin{aligned}
        &= \IP{\diff t f - \diff t[2] u}{w} + \IP{\diff t[2] (\tProjection u + \sum_{\nu=3}^{j+1}\addCorrector^{\nu-2}(\diff t[\nu-2] \tProjection u))}{w}[\LL][\Big]\\
        &\lesssim H^{p+2} \Norm{\diff t f}[H^k(\mathcal{T}_H)]\,\Norm{\nabla w}\\
        &\quad + H^2\Norm{\nabla \diff t [2] (u-\tProjection u - \sum_{\nu=3}^{j+1}\addCorrector^{\nu-2}(\diff t[\nu-2] \tProjection u))}[\LL][\Big]\, \Norm{\nabla w}. 
    \end{aligned}
  \end{multline*}
  Employing the choice~$w=\diff t (u-\tProjection u - \sum_{\nu=2}^{j+1}\addCorrector^{\nu-1}(\diff t[\nu-1] \tProjection u))$ and using~$\Norm{\nabla w}^2\lesssim \HIP{w}{w}$ yields 
  \begin{multline}\label{eq:dtu_hPu_first_step_estimate}
      \Norm{\nabla \diff t \big(u-\tProjection u - \sum_{\nu=2}^{j+1}\addCorrector^{\nu-1}(\diff t[\nu-1] \tProjection u)\big)}[\LL][\Big] \\
      \lesssim H^{p+2} \Norm{\diff t f}[H^k(\mathcal{T}_H)] + H^2\Norm{\nabla \diff t [2] (u-\tProjection u - \sum_{\nu=3}^{j+1}\addCorrector^{\nu-2}(\diff t[\nu-2] \tProjection u))}[\LL][\Big]
  \end{multline}
  The first term on the right-hand side has the optimal rate, and for the second term we can apply the argument in~\eqref{eq:dtu_hPu_first_step_estimate} recursively until the final estimate reads 
  \begin{multline*}
    \HIP{\diff t[j]\Big(u - \tProjection u - \sum_{\nu=j+1}^{j+1}\addCorrector^{\nu - j} (\diff t[\nu - j] \tProjection u)\Big)}{w}[\Big] = \IP{\diff t[j] f}{w} - \IP{\diff t[j+1](u-\tProjection u)}{w}\\
      \begin{aligned}
        &\lesssim H^{p+2} \Norm{\diff t[j] f}[H^k(\mathcal{T}_H)]\, \Norm{\nabla w} + H^2\Norm{\nabla \diff t[j+1] u}\, \Norm{\nabla w}. 
      \end{aligned}
  \end{multline*}
  If we choose~$w = \diff t[j] (u-\tProjection u - \addCorrector(\diff t \tProjection u))$, we obtain (with~$\Norm{\nabla w}^2\lesssim \HIP{w}{w}$)
  \begin{equation}\label{eq:dtu_hPu_final_step_estimate}
      \Norm{\nabla \diff t[j] (u - \tProjection u - \addCorrector(\diff t \tProjection u))} \lesssim H^{p+2} \Norm{\diff t[j] f}[H^k(\mathcal{T}_H)] + H^2\Norm{\nabla \diff t[j+1] u}
  \end{equation}
  Starting from estimate~\eqref{eq:dtu_hPu_first_step_estimate} and applying the recursion we obtain employing~\eqref{eq:dtu_hPu_final_step_estimate} 
  \begin{multline}\label{eq:dtu_hPu_estimate_complete}
      \Norm{\nabla \diff t (u-\tProjection u - \sum_{\nu=2}^{j+1}\addCorrector^{\nu-1}(\diff t[\nu-1] \tProjection u))}[\LL][\Big]\\
      \begin{aligned}
          &\lesssim H^{p+2}\sum_{\nu=2}^{j+1} H^{2(\nu-2)} \Norm{\diff t[\nu-1] f}[H^k(\mathcal{T}_H)] + H^{2j} \Norm{\nabla \diff t[j+1] u}. 
      \end{aligned}
  \end{multline}
  The estimate~\eqref{eq:dtu_hPu_estimate_complete} can now be used in~\eqref{eq:phi_estimate_intermediate} which yields 
  \begin{multline}\label{eq:phi_estimate_complete}
  \Norm{\nabla\varphi} \\
  \begin{aligned}
    &\lesssim H^{p+2}\Norm{f}[H^k(\mathcal{T}_H)] + H^{2j+2}\Norm{\nabla \diff t[j+1] u} + H^{p+4}\sum_{\nu=2}^{j+1} H^{2(\nu-2)} \Norm{\diff t[\nu-1] f}[H^k(\mathcal{T}_H)]\\
    &\lesssim (H^{p+2} + H^{2j+2})C_\mathrm{data}.
  \end{aligned}
  \end{multline}
  Finally, we are able to put all estimates together. Starting from~\eqref{eq:error_estimate_philoc}, using~\eqref{eq:philoc_gradient_estimate}, and~\eqref{eq:philoc_estimate} in the first estimate, we can apply the localization estimate~\eqref{eq:cPloc_estimate_complete} (where we used the specific choices for each~$\lambda_G$) and the mapping estimate~\eqref{eq:phi_estimate_complete} in the second estimate to obtain
  \begin{equation*}
      \begin{alignedat}{2}
          \Norm{e(t)} + \Norm{\nabla e(t)} &\lesssim_{T,p} && \sup_{s\in[0,T]} \Big[ \Norm{\nabla \varphi(s)} + \Norm{\nabla \diff t \varphi(s)}\\
          & && + \Norm{\nabla ((\tProjection + \hProjection)u(s)-(\tProjectionloc + \hProjectionloc)u(s))}\\
          & && + \Norm{\nabla \diff t ((\tProjection + \hProjection)u(s)-(\tProjectionloc + \hProjectionloc)u(s))}\Big]\\
          &\lesssim_{T,p} && \big[H^{p+2} + H^{2j+2} + \exp^{-C\ell}\big]C_\mathrm{data}.
      \end{alignedat}
  \end{equation*}
  The optimal error convergence follows now with the choices for~$j=\lceil \frac{p}{2} \rceil$, and~$\ell\sim C_p|\log H|$ in the theorem 
  \begin{equation*}
    \sup_{t\in[0,T]} \big[\Norm{e(t)} + \Norm{\nabla e(t)}\big]\lesssim_{T,p} H^{p+2}\, C_{\mathrm{data}}. 
  \end{equation*}
\end{proof}

\begin{remark}\label{rem:classical_holod}
    In the case, where~$k=1$, then by~\Cref{rem:parameters} it is sufficient to choose~$p=j=0$. This case is not directly covered in the proof, however it is easily adapted in the following way. The first part of the proof essentially works similarly until we arrive at~\eqref{eq:philoc_estimate} 
    \begin{equation*}
    \begin{aligned}
      \Norm{\nabla \varphi^\mathrm{loc}} &\leq \Norm{\nabla \varphi} + \Norm{\nabla ((\tProjection + \hat{\mathcal{P}}^{0})u-(\tProjectionloc + \hat{\mathcal{P}}^{0,\mathrm{loc}})u)}\\ 
      & = \Norm{\nabla \varphi} + \Norm{\nabla ((\Corrector^{[\ell]} - \Corrector)u)}.
    \end{aligned}
    \end{equation*}
    We have~$\varphi = u - \tProjection u$ and the optimal convergence~$r=p+2=2$ then follows from~\eqref{eq:tprojection_error} and~\Cref{lem:C_decay_error} with the choice~$\ell\sim C_p|\log H|$. 
\end{remark}

{\begin{remark}
    In the proofs in \Cref{sec:proofs}, we do not explicitly track the dependence on the lower and upper bound $\alpha$ and $\beta$ on the coefficient. We emphasize, however, that the constants in the decay results have a linear scaling with respect to the contrast $\beta/\alpha$ (which is a direct consequence of the results in~\cite{DonHM23}), while the pre-factors of the final localization error of the eho-LOD method involve a polynomial scaling. More precisely, by the iterative arguments in the proofs, we obtain the scaling $(\tfrac{1}{\alpha})^j(\tfrac\beta\alpha)^{2j+2}$. Note that this is not surprising as such scalings appear also in the estimates of classical LOD-type methods (i.e., $j=0$), amplified here by the additional corrections. As for classical results, the scaling does not appear to be problematic numerically.
\end{remark}}


\section{Numerical examples}\label{sec:numerical_examples}

\subsection{Time discretization}\label{subsec:time_discretization}
So far, we have mainly looked into the spatial discretization and the semi-discrete parabolic model problem. In the numerical experiments in this section, we apply the fourth-order backward difference formula (BDF4) for the temporal discretization. The main motivation is that since the method is high-order in space, we use a suitable high-order method in time to obtain a sufficiently accurate solution. In principle, a lower order method can be used, for example the Crank-Nicolson method, but the time step size~$\tau$ should be chosen sufficiently small so that the spatial error is dominant. Let $\tilde{\partial}_t^n, \;n = 1,\cdots,4$ be the $n$-th order backward difference formula approximating the time derivative $\partial_t$. The eho-LOD-BDF4 seeks a series of solutions~$\{{\cu^n}\}_{n=0}^N\subset \cVloc$ such that
\begin{subequations}\label{eq:eholod_bdf4}
\begin{align}    
    \tau^{-1} \IP{\tilde{\partial}_t^n\cu^{n}}{\cv} + \HIP{\cu^{n} }{\cv} &= \IP{f(t_{n})}{\cv}, \quad n = 1,2,3,\\
    \tau^{-1} \IP{\tilde{\partial}^4\cu^{n}}{\cv} + \HIP{\cu^{n} }{\cv} &= \IP{f(t_{n})}{\cv}, \quad n \ge 4,
\end{align}
\end{subequations}
for all~$\cv\in\cVloc$, where the initial condition is given by $\cu^0 = \tProjectionloc u_0 + \hProjectionloc u_0$. We use the BDF4 method to discretize in time, and we use the BDF$n$ method ($n$=1,2,3) to approximate the starting values. Note that this might lead to an order reduction of the temporal error if the first time steps are not chosen sufficiently small. Since the focus of this work is a novel spatial discretization, we apply the same temporal discretization for the~eho-LOD solution~$\cu^n$ and the reference solution~$u_h$. Therefore, only the spatial error is observed in our experiments. For a more thorough time discretization, one may choose any of the methods discussed in~\cite[\S~III.1,~p.~356]{hairer1993solving}. In our simulations, we use~$\tau=2^{-9}$ in one dimension and $\tau = 2^{-8}$ in two dimensions.

\subsection{Numerical experiments}

\begin{figure}[!t]
    \centering
    \includegraphics[scale=.5]{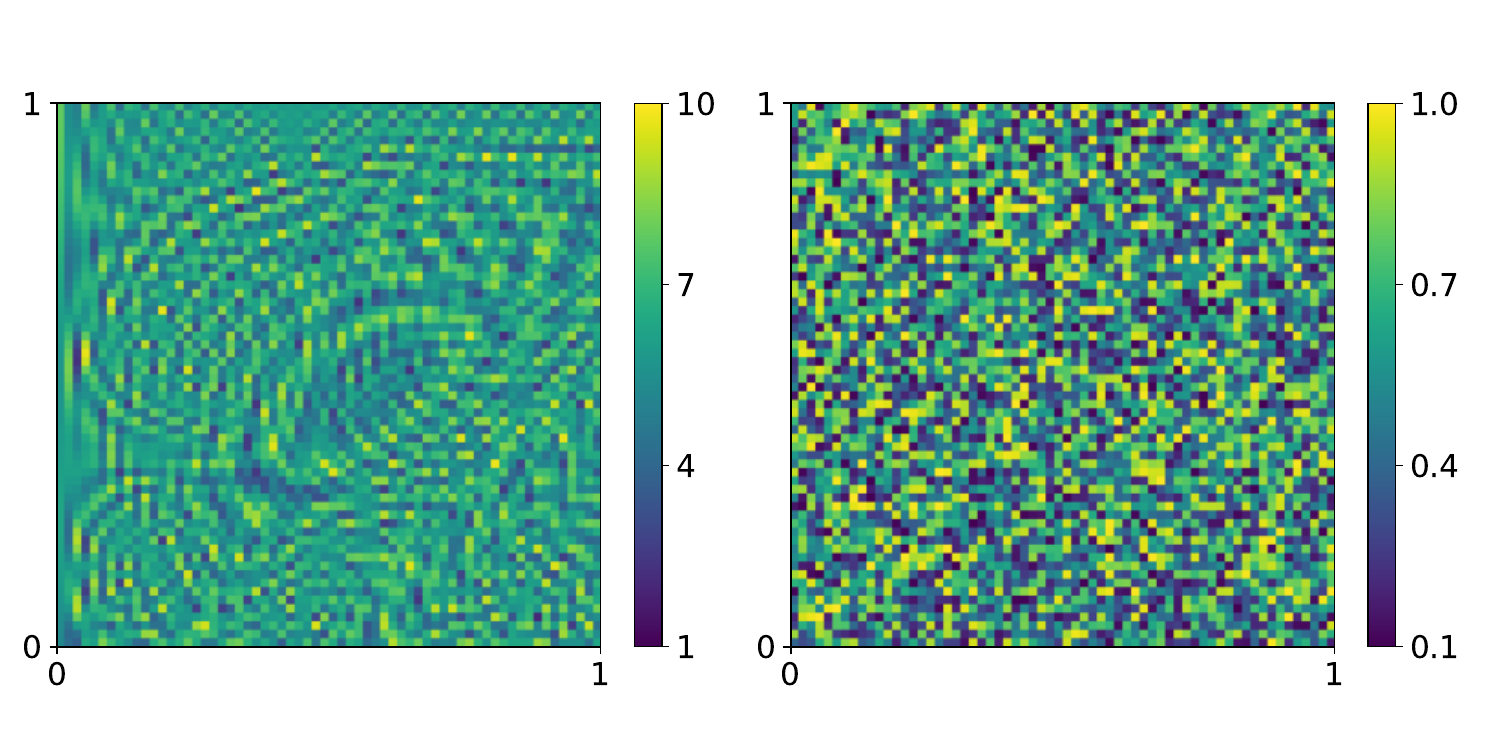}
    \vspace*{-0.5cm}
    \caption{\small Coefficient with some coarse structure and fine oscillations used for~\Cref{ex:example1} (left) and random diffusion coefficient for~\Cref{ex:example2} (right).}
    \label{fig:coefficient}
\end{figure}

{In this section, we present some numerical examples to verify the theoretical results. For the construction of a basis of $\cVloc$, we discretize the local problems using the  finite element space~$V_h$ with mesh size~$h\leq\varepsilon$. We present examples in 1D and 2D, where we consider the computational domain $\Omega = (0,1)^d$ with $d=1,2$. We compute a reference solution on the fine mesh $\mathcal{T}_h$ and measure the error at the final time $T=1$ in the energy norm~$|\cdot| = \sqrt{a(\cdot, \cdot)}$. %
To verify our theoretical results, we use quad-precision arithmetic and direct solvers in our calculations for examples in one dimension. However, high-precision arithmetic does not scale well in higher dimensions, and so for~$d=2$, we use an iterative solver based on the Schur complement with an appropriate preconditioner for solving the linear system (see~\Cref{rem:practical_split}). The code is written using the Julia programming language (\cite{Julia-2017}) and can be found online in the Github repository \href{https://github.com/Balaje/MultiScaleFEM.jl}{\texttt{https://github.com/Balaje/MultiScaleFEM.jl}}. The finite element computations use the \texttt{Gridap.jl} library (\cite{Badia2020, Verdugo2022}) and the iterative solvers with preconditioners from the \texttt{IterativeSolvers.jl} package.}%

\begin{example}\label{ex:example1}
    We consider the diffusion coefficient shown in~\Cref{fig:coefficient} (left), which oscillates on the scale $2^{-6}$, and study the convergence rates of the eho-LOD-BDF4 method with $\ell = \infty$. The fine-scale space $V_h$ is constructed on a uniform mesh with $h = 2^{-8}$, sufficient to resolve the oscillations in the diffusion coefficient. We consider the source term~$f(x,t) = 20\pi^2 \sin{\left(\pi x_1\right)}\sin{\left(\pi x_2\right)}\sin^5(t)$ and zero initial data so that Assumption \ref{ass:regularity} is satisfied for optimal convergence with~$p\leq 4$.  We discretize in time using the BDF4 method \eqref{eq:eholod_bdf4} with a constant time step size $\tau = 2^{-8}$. We also compute the reference solution using the same time discretization scheme since we are interested in tracking the spatial error only.

    In~\Cref{fig:example1}, the errors of the enriched higher-order LOD method are plotted against the mesh size~$H$ and the degrees of freedom. For~$p=0$, the optimal second order rate can be observed for~$j=0$. Enriched spaces do not pay off in terms of convergence rates as the rate is already optimal. Only the size of the error is positively affected. 
    For~$p=1$, we can observe a similar behavior, where for~$j=1,2$ we have only an improvement in the error and not the rate. We have thus omitted the plots for better visibility. Further, the plot for~$j=0$ has third order convergence, which comes from the fact that the second order plateau (which shows for~$p=2,3,4$) is not reached yet. For~$p=2,3,4$ the error plots reach the second order plateau for smaller mesh sizes, and choosing~$j=1$ for~$p=2$ and~$j=2$ for~$p=3,4$ yields the optimal error rate~$r=p+2$. In addition, we can also observe some super-convergence for~$j=1$ and $p=3,4$ for larger~$H$, which then saturates, and it can be observed that~$j=2$ is necessary. 
    {Similar to~$j=1$ for~$p=4$, the same is observed for $p=0$. We see that the errors are smaller for~$j=1,2$ compared to~$j=0$, and as $H$ decreases, the rates saturate and the errors approach the errors of the case when $j=0$.}  Thus, we can conclude that the choice $j=\lceil \tfrac{p}{2}\rceil$ in Theorem \ref{thm:semi_discrete} yields optimal convergence rates, a larger $j$ may improve the error but not convergence rate, and a smaller $j$ may lead to suboptimal convergence rate or larger errors. In Figure~\ref{fig:example1} (right), we can clearly observe that in terms of number the degrees of freedom increasing $j$ and $p$ appropriately does really improve convergence. 
\end{example}

\begin{figure}[!t]
  \centering
  \begin{subfigure}[b]{.5\textwidth}
    \pgfplotscreateplotcyclelist{rates}{%
blue,ultra thick,solid\\%
green01270,ultra thick,dashed\\%
red,ultra thick,densely dotted\\%
violet,ultra thick,loosely dotted\\%
darkorange25512714,ultra thick,dashdotted\\%
}

\pgfplotscreateplotcyclelist{ell}{%
every mark/.append style={fill=white,solid},mark=square*,mark size=4\\%
every mark/.append style={fill=white,solid},mark=triangle*,mark size=6\\%
every mark/.append style={fill=white,solid},mark=diamond*,mark size=6\\%
}

\begin{tikzpicture}

\definecolor{darkgray176}{RGB}{176,176,176}
\definecolor{darkturquoise0191191}{RGB}{0,191,191}
\definecolor{green01270}{RGB}{0,127,0}
\definecolor{darkorange25512714}{RGB}{255,127,14}

\begin{axis}[
cycle multi list={ell\nextlist
									rates},
width=12cm,
height=9cm,
axis line style=very thick,
log basis x={10},
log basis y={10},
tick align=center,
tick pos=left,
xlabel={{\Large mesh size $H\phantom{\mathrlap{\mathcal{N} = (j+1)(\frac{p+1}{H})^d}}$}},
x label style={at={(axis description cs:0.5,-0.02)},anchor=north},
x grid style={darkgray176},
xmin=0.058, xmax=0.53,
xmode=log,
xtick style={color=black},
xtick={0.0625,0.125,0.25,0.5},
xticklabels={
  \(\displaystyle {2^{-4}}\),
  \(\displaystyle {2^{-3}}\),
  \(\displaystyle {2^{-2}}\),
  \(\displaystyle {2^{-1}}\),
},
ylabel={{\Large rel. error in $|\cdot |_{a}$}},
y label style={at={(axis description cs:-0.001,0.5)},anchor=south},
y grid style={darkgray176},
ymin=6e-12, ymax=1.5,
ymode=log,
ytick style={color=black},
ytick={1e-11,1e-10,1e-08,1e-06,0.0001,0.01,1},
yticklabels={
  \(\displaystyle {10^{-11}}\),
  \(\displaystyle {10^{-10}}\),
  \(\displaystyle {10^{-8}}\),
  \(\displaystyle {10^{-6}}\),
  \(\displaystyle {10^{-4}}\),
  \(\displaystyle {10^{-2}}\),
  \(\displaystyle {10^{0}}\),
},
ticklabel style = {font=\Large},
]

\addplot table[x=H,y=p0j0,col sep=comma]{./Images/errors.csv};
\addplot table[x=H,y=p1j0,col sep=comma]{./Images/errors.csv};
\addplot table[x=H,y=p2j0,col sep=comma]{./Images/errors.csv};
\addplot table[x=H,y=p3j0,col sep=comma]{./Images/errors.csv};
\addplot table[x=H,y=p4j0,col sep=comma]{./Images/errors.csv};

\addplot table[x=H,y=p0j1,col sep=comma]{./Images/errors.csv};
\pgfplotsset{cycle list shift=1}
\addplot table[x=H,y=p2j1,col sep=comma]{./Images/errors.csv};
\addplot table[x=H,y=p3j1,col sep=comma]{./Images/errors.csv};
\addplot table[x=H,y=p4j1,col sep=comma]{./Images/errors.csv};

\addplot table[x=H,y=p0j2,col sep=comma]{./Images/errors.csv};
\pgfplotsset{cycle list shift=2}
\addplot table[x=H,y=p2j2,col sep=comma]{./Images/errors.csv};
\addplot table[x=H,y=p3j2,col sep=comma]{./Images/errors.csv};
\addplot table[x=H,y=p4j2,col sep=comma]{./Images/errors.csv};

\addplot [very thick, black]
table {
  0.5 0.9128209687408646
  0.0625 0.01426282763657601
};
\addplot [very thick, black]
table {
  0.5 0.1285081022944572
  0.0625 0.0002509923872938617
};
\addplot [very thick, black]
table {
  0.5 0.0003324236240996395
  0.0625 6.492648908196084e-07
};
\addplot [very thick, black]
table {
  0.5 2.3693534805812444e-5
  0.0625 7.23069299493788e-10
};
\addplot [very thick, black]
table {
  0.5 1.4513536174900081e-6
  0.0625 5.536474676094086e-12
};

\end{axis}

\end{tikzpicture}
  \end{subfigure}%
  \begin{subfigure}[b]{.5\textwidth}
    \pgfplotscreateplotcyclelist{rates}{%
blue,ultra thick,solid\\%
green01270,ultra thick,dashed\\%
red,ultra thick,densely dotted\\%
violet,ultra thick,loosely dotted\\%
darkorange25512714,ultra thick,dashdotted\\%
}

\pgfplotscreateplotcyclelist{ell}{%
every mark/.append style={fill=white,solid},mark=square*,mark size=4\\%
every mark/.append style={fill=white,solid},mark=triangle*,mark size=6\\%
every mark/.append style={fill=white,solid},mark=diamond*,mark size=6\\%
}

\begin{tikzpicture}

\definecolor{darkgray176}{RGB}{176,176,176}
\definecolor{darkturquoise0191191}{RGB}{0,191,191}
\definecolor{green01270}{RGB}{0,127,0}
\definecolor{darkorange25512714}{RGB}{255,127,14}

\begin{axis}[
cycle multi list={ell\nextlist
									rates},
width=12cm,
height=9cm,
axis line style=very thick,
log basis x={10},
log basis y={10},
tick align=center,
tick pos=left,
legend pos=south west,
xlabel={{\Large degrees of freedom~$\mathcal{N} = (j+1)(\frac{p+1}{H})^d$}},
x label style={at={(axis description cs:0.5,-0.02)},anchor=north},
x grid style={darkgray176},
xmin=3, xmax=2.3e+04,
xmode=log,
xtick style={color=black},
xtick={10,1e+02,1e+03,1e+04},
xticklabels={
  \(\displaystyle {10^{1}}\),
  \(\displaystyle {10^{2}}\),
  \(\displaystyle {10^{3}}\),
  \(\displaystyle {10^{4}}\)
},
y grid style={darkgray176},
ymin=6e-12, ymax=1.5,
ymode=log,
ytick style={color=black},
ytick={1e-11,1e-10,1e-08,1e-06,0.0001,0.01,1},
yticklabels={
  \(\displaystyle {10^{-11}}\),
  \(\displaystyle {10^{-10}}\),
  \(\displaystyle {10^{-8}}\),
  \(\displaystyle {10^{-6}}\),
  \(\displaystyle {10^{-4}}\),
  \(\displaystyle {10^{-2}}\),
  \(\displaystyle {10^{0}}\),
},
ticklabel style = {font=\Large},
]
\addlegendimage{no markers,blue,solid,ultra thick}\addlegendentry{{\large $\,p=0$}}
\addlegendimage{no markers,green01270,densely dotted,ultra thick}\addlegendentry{{\large $\,p=1$}}
\addlegendimage{no markers,red,dashed,ultra thick}\addlegendentry{{\large $\,p=2$}}
\addlegendimage{no markers,violet,loosely dotted,ultra thick}\addlegendentry{{\large $\,p=3$}}
\addlegendimage{no markers,darkorange25512714,dashdotted,ultra thick}\addlegendentry{{\large $\,p=4$}}

\addlegendimage{no markers,black,thick}\addlegendentry{{\large $\,r=2$}}
\addlegendimage{no markers,black,thick}\addlegendentry{{\large $\,r=3$}}
\addlegendimage{no markers,black,thick}\addlegendentry{{\large $\,r=4$}}
\addlegendimage{no markers,black,thick}\addlegendentry{{\large $\,r=5$}}
\addlegendimage{no markers,black,thick}\addlegendentry{{\large $\,r=6$}}

\addlegendimage{only marks,every mark/.append style={fill=white}, mark=square*}\addlegendentry{{\large $\,j=0$}}
\addlegendimage{only marks,every mark/.append style={fill=white}, mark=triangle*}\addlegendentry{{\large $\,j=1$}}
\addlegendimage{only marks,every mark/.append style={fill=white}, mark=diamond*}\addlegendentry{{\large $\,j=2$}}

\addplot table[x=p0j0dofs,y=p0j0,col sep=comma]{./Images/errors_DoFs.csv};
\addplot table[x=p1j0dofs,y=p1j0,col sep=comma]{./Images/errors_DoFs.csv};
\addplot table[x=p2j0dofs,y=p2j0,col sep=comma]{./Images/errors_DoFs.csv};

\pgfplotsset{cycle list shift=4}
\addplot table[x=p2j1dofs,y=p2j1,col sep=comma]{./Images/errors_DoFs.csv};
\addplot table[x=p3j1dofs,y=p3j1,col sep=comma]{./Images/errors_DoFs.csv};
\addplot table[x=p4j1dofs,y=p4j1,col sep=comma]{./Images/errors_DoFs.csv};

\pgfplotsset{cycle list shift=7}
\addplot table[x=p3j2dofs,y=p3j2,col sep=comma]{./Images/errors_DoFs.csv};
\addplot table[x=p4j2dofs,y=p4j2,col sep=comma]{./Images/errors_DoFs.csv};


\end{axis}

\end{tikzpicture}
  \end{subfigure}
  \caption{\small Relative energy errors for~\Cref{ex:example1} with respect to the mesh size (left) and degrees of freedom (right). The legend applies to both plots. }\label{fig:example1}
\end{figure}
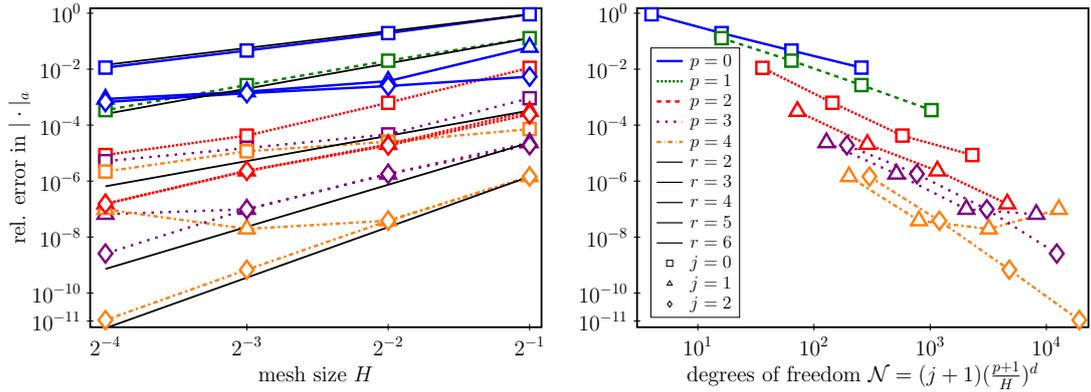

\begin{example}\label{ex:example2}
    We consider the diffusion coefficient shown in the right-hand side of~\Cref{fig:coefficient}, which oscillates randomly on the scale $2^{-6}$, and study localization effects for the eho-LOD-BDF4 method. We generate the random diffusion coefficient shown in~\Cref{fig:coefficient}  (right) with values between $0.1$ and $1$. We use the same right-hand side $f$ as in~\Cref{ex:example1} as well as the discretization parameters $h = 2^{-8}$ and $\tau = 2^{-8}$.

    The error plot on the left-hand side of~\Cref{fig:localization} shows the effect of the localization for the enriched corrections. Here, we can observe that a localization parameter~$\ell=9$ is sufficient to show the optimal convergence rate~$r=5$ for~$p=3$. For reference, the top line represents the same polynomial degree~$p=3$ with~$j=0$ for~$\ell=4$, and we can observe that with the same localization the enriched version with~$j=2$ outperforms the classical higher-order LOD which confirms the localization result from~\Cref{thm:localization}. 
\end{example}

\begin{example}\label{ex:example3}
    We consider a one-dimensional example with a random diffusion coefficient  oscillating between $0.1$ and $1$ at the scale $\varepsilon = 2^{-8}$. We choose
    \begin{equation*}
    f(x,t) = \begin{cases}
        0, & x < 0.5\\
        \sin{\left( \pi x \right)}\sin^5{(t)}, & x \ge 0.5.
    \end{cases}
    \end{equation*}
    We construct the fine-scale space~$V_h$ on a uniform mesh with $h = 2^{-10}$ and solve the problem using the BDF4 method in time with $\tau = 2^{-9}$. Finally, we use the fine-scale discretization $V_h$ to compute the reference solution.

    In this one-dimensional example, we showcase the polynomial degrees~$p=2,3,4$ with their respective optimal~$j$ as well as the errors for $j=0$ in Figure~\ref{fig:localization} (right). 
    For each point the optimal localization parameter is portrayed. If we consider the optimal localization for~$p=4$ with~$j=0$ at~$H=2^{-4}$, then we have that~$6=\ell\gtrsim 2|\log H|$, since we can only expect second order of convergence. In contrast we have for~$j=2$ an optimal localization of~$14 = \ell\gtrsim 6|\log H|$ for convergence of order~$r=6$. This is perfectly in line with our theory derived in~\cref{sec:num_hom}, where we naturally need larger localization for much lower errors but the constant when choosing the localization parameter scales similarly for all~$j$. 
\end{example}

\begin{figure}[!t]
  \centering
  \begin{subfigure}[b]{.46\textwidth}
    \pgfplotscreateplotcyclelist{rates}{%
blue,ultra thick,solid\\%
green01270,ultra thick,dashed\\%
red,ultra thick,densely dotted\\%
violet,ultra thick,loosely dotted\\%
darkorange25512714,ultra thick,loosely dashdotted\\%
cyan,ultra thick,loosely dashed\\%
}

\pgfplotscreateplotcyclelist{enrich}{%
every mark/.append style={fill=white,solid},mark=square*,mark size=4\\%
every mark/.append style={fill=white,solid},mark=triangle*,mark size=6\\%
every mark/.append style={fill=white,solid},mark=diamond*,mark size=6\\%
}

\begin{tikzpicture}

\definecolor{darkgray176}{RGB}{176,176,176}
\definecolor{darkturquoise0191191}{RGB}{0,191,191}
\definecolor{green01270}{RGB}{0,127,0}
\definecolor{darkorange25512714}{RGB}{255,127,14}
\definecolor{lava}{rgb}{0.81, 0.06, 0.13}

\begin{axis}[
cycle multi list={enrich\nextlist
									rates},
width=12cm,
height=9cm,
axis line style=very thick,
log basis x={10},
log basis y={10},
tick align=center,
tick pos=left,
legend pos=south east,
xlabel={{\Large mesh size $H$}},
x label style={at={(axis description cs:0.5,-0.02)},anchor=north},
x grid style={darkgray176},
xmin=0.06, xmax=0.52,
xmode=log,
xtick style={color=black},
xtick={0.0625,0.125,0.25,0.5},
xticklabels={
  \(\displaystyle {2^{-4}}\),
  \(\displaystyle {2^{-3}}\),
  \(\displaystyle {2^{-2}}\),
  \(\displaystyle {2^{-1}}\),
},
ylabel={{\Large rel. error in $|\cdot |_{a}$}},
y label style={at={(axis description cs:-0.001,0.5)},anchor=south},
y grid style={darkgray176},
ymin=5e-9, ymax=1e-2,
ymode=log,
ytick style={color=black},
ytick={1e-08,1e-07,1e-06,1e-05,1e-04,1e-03,1e-02},
yticklabels={
  \(\displaystyle {10^{-8}}\),
  \(\displaystyle {10^{-7}}\),
  \(\displaystyle {10^{-6}}\),
  \(\displaystyle {10^{-5}}\),
  \(\displaystyle {10^{-4}}\),
  \(\displaystyle {10^{-3}}\),
  \(\displaystyle {10^{-2}}\),
},
ticklabel style = {font=\Large},
]
\addlegendimage{no markers,blue,solid,ultra thick}\addlegendentry{$\ell=4$}
\addlegendimage{no markers,green01270,densely dotted,ultra thick}\addlegendentry{$\ell=5$}
\addlegendimage{no markers,red,dashed,ultra thick}\addlegendentry{$\ell=6$}
\addlegendimage{no markers,violet,loosely dotted,ultra thick}\addlegendentry{$\ell=7$}
\addlegendimage{no markers,darkorange25512714,dashdotted,ultra thick}\addlegendentry{$\ell=8$}
\addlegendimage{no markers,cyan,loosely dashed,ultra thick}\addlegendentry{$\ell=9$}
\addlegendimage{only marks,every mark/.append style={fill=white}, mark=square*}\addlegendentry{$j=0$}
\addlegendimage{only marks,every mark/.append style={fill=white}, mark=diamond*}\addlegendentry{$j=2$}
\addlegendimage{no markers,black,thick}\addlegendentry{$r=5$}

\addplot table[x=H,y=ell4j0,col sep=comma]{./Images/errors_2d_ell.csv};


\pgfplotsset{cycle list shift=11}
\addplot table[x=H,y=ell4j2,col sep=comma]{./Images/errors_2d_ell.csv};
\addplot table[x=H,y=ell5j2,col sep=comma]{./Images/errors_2d_ell.csv};
\addplot table[x=H,y=ell6j2,col sep=comma]{./Images/errors_2d_ell.csv};
\addplot table[x=H,y=ell7j2,col sep=comma]{./Images/errors_2d_ell.csv};
\addplot table[x=H,y=ell8j2,col sep=comma]{./Images/errors_2d_ell.csv};
\addplot table[x=H,y=ell9j2,col sep=comma]{./Images/errors_2d_ell.csv};

\addplot [very thick, black]
table {
  0.5 0.00023379480389804002
  0.0625 7.134851193177491e-09
};

\end{axis}

\end{tikzpicture}
  \end{subfigure}%
  \begin{subfigure}[b]{.54\textwidth}
    \pgfplotscreateplotcyclelist{rates}{%
  scatter,dashed,red,ultra thick,every mark/.append style={fill=white,solid},point meta=explicit symbolic,scatter/classes={5={mark=*,mark size=4},6={mark=square*,mark size=4},7={mark=triangle*,mark size=6},8={mark=diamond*,mark size=6},9={mark=pentagon*,mark size=6},10={mark=x,mark size=6},11={mark=+,mark size=6},12={mark=asterisk,mark size=6},14={mark=|,mark size=6},inf={mark=10-pointed star,mark size=6}}\\%
  scatter,densely dotted,violet,ultra thick,every mark/.append style={fill=white,solid},point meta=explicit symbolic,scatter/classes={5={mark=*,mark size=4},6={mark=square*,mark size=4},7={mark=triangle*,mark size=6},8={mark=diamond*,mark size=6},9={mark=pentagon*,mark size=6},10={mark=x,mark size=6},11={mark=+,mark size=6},12={mark=asterisk,mark size=6},14={mark=|,mark size=6},inf={mark=10-pointed star,mark size=6}}\\%
  scatter,loosely dotted,darkorange25512714,ultra thick,every mark/.append style={fill=white,solid},point meta=explicit symbolic,scatter/classes={5={mark=*,mark size=4},6={mark=square*,mark size=4},7={mark=triangle*,mark size=6},8={mark=diamond*,mark size=6},9={mark=pentagon*,mark size=6},10={mark=x,mark size=6},11={mark=+,mark size=6},12={mark=asterisk,mark size=6},14={mark=|,mark size=6},inf={mark=10-pointed star,mark size=6}}\\%
}

\begin{tikzpicture}

\definecolor{darkgray176}{RGB}{176,176,176}
\definecolor{darkturquoise0191191}{RGB}{0,191,191}
\definecolor{green01270}{RGB}{0,127,0}
\definecolor{darkorange25512714}{RGB}{255,127,14}

\begin{axis}[
cycle list name=rates,
width=12cm,
height=9cm,
axis line style=very thick,
log basis x={10},
log basis y={10},
tick align=center,
tick pos=left,
legend pos=outer north east,
xlabel={{\Large mesh size $H$}},
x label style={at={(axis description cs:0.5,-0.02)},anchor=north},
x grid style={darkgray176},
xmin=0.029, xmax=0.52,
xmode=log,
xtick style={color=black},
xtick={0.03125,0.0625,0.125,0.25,0.5},
xticklabels={
  \(\displaystyle {2^{-5}}\),
  \(\displaystyle {2^{-4}}\),
  \(\displaystyle {2^{-3}}\),
  \(\displaystyle {2^{-2}}\),
  \(\displaystyle {2^{-1}}\),
},
y grid style={darkgray176},
ymin=5e-16, ymax=3.1e-3,
ymode=log,
ytick style={color=black},
ytick={1e-15,1e-13,1e-11,1e-9,1e-7,1e-5,0.001},
yticklabels={
  \(\displaystyle {10^{-15}}\),
  \(\displaystyle {10^{-13}}\),
  \(\displaystyle {10^{-11}}\),
  \(\displaystyle {10^{-9}}\),
  \(\displaystyle {10^{-7}}\),
  \(\displaystyle {10^{-5}}\),
  \(\displaystyle {10^{-3}}\),
},
ticklabel style = {font=\Large},
]
\addlegendimage{no markers,red,dashed,ultra thick}\addlegendentry{$p=2$}
\addlegendimage{no markers,violet,loosely dotted,ultra thick}\addlegendentry{$p=3$}
\addlegendimage{no markers,darkorange25512714,dashdotted,ultra thick}\addlegendentry{$p=4$}

\addlegendimage{only marks,every mark/.append style={fill=white}, mark=*}\addlegendentry{$\ell=5$}
\addlegendimage{only marks,every mark/.append style={fill=white}, mark=square*}\addlegendentry{$\ell=6$}
\addlegendimage{only marks,every mark/.append style={fill=white}, mark=triangle*}\addlegendentry{$\ell=7$}
\addlegendimage{only marks,every mark/.append style={fill=white}, mark=diamond*}\addlegendentry{$\ell=8$}
\addlegendimage{only marks,every mark/.append style={fill=white}, mark=pentagon*}\addlegendentry{$\ell=9$}
\addlegendimage{only marks, mark=x}\addlegendentry{$\ell=10$}
\addlegendimage{only marks, mark=+}\addlegendentry{$\ell=11$}
\addlegendimage{only marks, mark=asterisk}\addlegendentry{$\ell=12$}
\addlegendimage{only marks, mark=|}\addlegendentry{$\ell=14$}
\addlegendimage{only marks, mark=10-pointed star}\addlegendentry{$\ell=\infty$}

\addlegendimage{no markers,black,thick}\addlegendentry{$r=2$}
\addlegendimage{no markers,black,thick}\addlegendentry{$r=4$}
\addlegendimage{no markers,black,thick}\addlegendentry{$r=5$}
\addlegendimage{no markers,black,thick}\addlegendentry{$r=6$}

\addplot table[x=H,y=p2j0,col sep=comma,meta=p2j0ell]{./Images/errors_1d_ell.csv};
\addplot table[x=H,y=p3j0,col sep=comma,meta=p3j0ell]{./Images/errors_1d_ell.csv};
\addplot table[x=H,y=p4j0,col sep=comma,meta=p4j0ell]{./Images/errors_1d_ell.csv};

\addplot table[x=H,y=p2j1,col sep=comma,meta=p2j1ell]{./Images/errors_1d_ell.csv};

\addplot table[x=H,y=p3j2,col sep=comma,meta=p3j2ell]{./Images/errors_1d_ell.csv};
\addplot table[x=H,y=p4j2,col sep=comma,meta=p4j2ell]{./Images/errors_1d_ell.csv};

\addplot [very thick, black]
table {
  0.5 0.0006714810858104163
  0.25 1.67870271452604067835341480670564915e-04
  0.03125 2.6229729914469387e-06
};
\addplot [very thick, black]
table {
  0.5 2.47082608022828132939487955224990418e-05
  0.03125 3.7701813968327047e-10
};
\addplot [very thick, black]
table {
  0.5 8.18530988921623297339729586327583121e-07
  0.03125 7.806119813171608e-13
};
\addplot [very thick, black]
table {
  0.5 3.08863881246208736396448160582880449e-08
  0.03125 1.8409721925628707e-15
};

\end{axis}

\end{tikzpicture}
  \end{subfigure}
  \caption{\small Relative energy errors for~\Cref{ex:example2} with random diffusion coefficient and~$p=3$ with different localization parameters (left). Localization errors for~\Cref{ex:example3} of a one-dimensional random diffusion coefficient (right). On the right, the first three lines at the top are the classical higher-order LOD with $j=0$, and the lines below that are~$j=1$ for~$p=2$ and then~$j=2$ for~$p=3,4$}\label{fig:localization}
\end{figure}
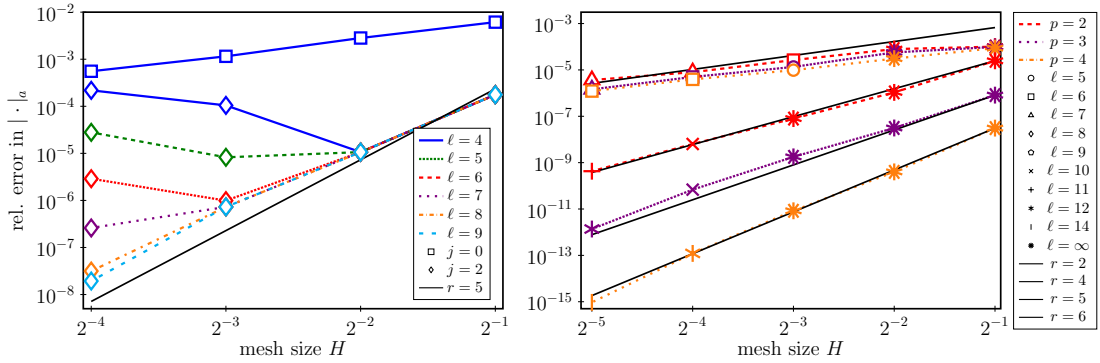

\section{Conclusion}\label{sec:conclusion}
In this paper, we have developed the enriched higher-order localized orthogonal decomposition method which yields optimal convergence rates in space for the heat equation with highly oscillatory coefficients. We have presented the construction of an enriched correction operator to  the higher-order multiscale space from the elliptic setting and proved optimal convergence rates without additional assumptions on the coefficient. In addition, we have also proved the exponential decay properties of the enriched corrections and showed that the developed method preserves the localization of the higher-order LOD method. Finally, we have presented numerical experiments to verify the theoretical results, and observed very good agreement. 

{The present work opens several directions for further investigation. As already mentioned,  the enrichment strategy appears well-suited for generalization to other time-dependent linear PDEs with an appropriate adjustment of the error analysis. In addition, it would be of interest to explore the performance of the proposed method in the context of time-dependent coefficients, for example by incorporating the adaptive update strategy for the multiscale basis functions from \cite{MaierVerfurth2022}. %
While these aspects lie beyond the scope of the current work, our results suggest that the proposed framework is a solid foundation for addressing these scenarios. } 

\section*{Acknowledgments}
B. Kalyanaraman was funded by the Kempe foundation in Sweden with Project-ID JCK22-0012. R.~Maier acknowledges funding by the Deutsche Forschungsgemeinschaft (DFG, German Research
Foundation) -- Project-ID 545165789. Part of the computation was carried out in Project hpc2n2024-109, hpc2n2025-217 and hpc2nstor2025-060  provided by the National Academic Infrastructure for Supercomputing in Sweden (NAISS), partially funded by the Swedish Research Council through grant agreement no.~2022-06725. The authors also acknowledge support by the state of Baden-Württemberg through bwHPC.


\newcommand{\etalchar}[1]{$^{#1}$}

\end{document}